\documentclass[12pt,a4paper]{scrartcl}
\usepackage[english]{babel}			% Dokument in englischer Sprache, AUX-Dateien loeschen
\usepackage{hyperref} 				% Hyperlink im PDF
\usepackage{paralist}				% Fuer compactenum-Umgebung
\usepackage{graphicx}				% Fuer Einfuegen von Bildern
\usepackage{soul}				% Highlighting, \hlc[red]{...}
\usepackage{mdframed}				% Farbige Boxen, \begin{mdframed}[linecolor=red,linewidth=3]...\end{mdframed}
\usepackage{amsfonts} 				% Schriftart-Erweiterungen (Fraktur, etc)
\usepackage{amsmath} 				% Mehr Optionen zur Formatierung etc von Formeln, z.B. fuer align-Umgebung
\usepackage{mathabx}				% Mathematische Symbole wie z.B. widehat. Benoetigt amsmath und amsfonts (zuerst einbinden)
\usepackage[arrow, matrix, curve]{xy} 		% Matrizen und so
\usepackage{amsopn}				% Fuer DeclareMathOperator
\usepackage{amsthm} 				% Theoremumgebung
\usepackage{extarrows}				% Lange Pfeile
\usepackage{mathtools}				% Fuer \xhookrightarrow 

	\DeclareMathOperator{\id}{id}
	
	\DeclareMathOperator{\ad}{ad}

	\DeclareMathOperator{\der}{der}
	\DeclareMathOperator{\im}{im}

	\DeclareMathOperator{\per}{per}
	\DeclareMathOperator{\GL}{GL}
	\DeclareMathOperator{\Alt}{Alt}
	\DeclareMathOperator{\Aut}{Aut}
	
	\DeclareMathOperator{\Ad}{Ad}
	\DeclareMathOperator{\pr}{pr}
	\DeclareMathOperator{\supp}{supp}

	\DeclareMathOperator{\ev}{ev}

	\newcommand{\N}{\mathbb{N}}
	\newcommand{\R}{\mathbb{R}}

	\newcommand{\ep}{\varepsilon}
	\renewcommand{\a}{\alpha}	%ersetzt ein Punkt-Symbol; Ungefaehrliche Ersetzung
	\renewcommand{\b}{\beta}	%ersetzt ein Strich-Symbol; Ungefaehrliche Ersetzung
	\newcommand{\g}{\gamma}
	\renewcommand{\gg}{\Gamma}	%ersetzt ein Symbol; Ungefaehrlich
	\renewcommand{\d}{\delta}	%ersetzt ein Punkt-Symbol; Ungefaehrliche Ersetzung
	\newcommand{\gd}{\Delta}
	
	\newcommand{\e}{\eta}
	\renewcommand{\t}{\theta}	%ersetzt ein Strich-Symbol; Ungefaehrliche Ersetzung
	
	\renewcommand{\k}{\kappa}	%ersetzt ein Strich-Symbol; Ungefaehrliche Ersetzung
	\renewcommand{\l}{\lambda}	%ersetzt ein Strich-Symbol; Ungefaehrliche Ersetzung
	\newcommand{\gl}{\Lambda}

	\newcommand{\p}{\pi}
	\newcommand{\gp}{\Pi}
	\renewcommand{\r}{\rho}		%ersetzt ein Kreis-Symbol; Ungefaehrliche Ersetzung
	\newcommand{\s}{\sigma}
	
	\newcommand{\ta}{\tau}
	\newcommand{\ph}{\varphi}
	\newcommand{\gph}{\Phi}
	\newcommand{\ps}{\psi}
	
	\renewcommand{\o}{\omega}		%ersetzt ein Kreis-Symbol; Ungefaehrliche Ersetzung 
	\newcommand{\go}{\Omega}

	\newcommand{\fg}{\frak{g}}
	\newcommand{\fh}{\frak{h}}

	\newcommand{\fz}{\frak{z}}

	\renewcommand{\le}{\left}
	\newcommand{\ri}{\right}
	\newcommand{\set}[1]{\le\{#1\ri\}}

	\newcommand{\ra}{\rightarrow}
	
	\newcommand{\hra}{\hookrightarrow}

	\newcommand{\xra}[1]{\xrightarrow{#1}}

	\newcommand{\ub}[2]{\underbracket{#1}_{#2}}

	\newcommand{\subs}{\subseteq}
	
	\newcommand{\ol}[1]{\overline{#1}}
	\newcommand{\ti}{\times}

	\newcommand{\til}{\tilde}
	\newcommand{\wtil}{\widetilde}

	\newcommand{\we}{\wedge}

	\newcommand{\hor}{\text{hor}}
	\newcommand{\fix}{\text{fix}}
	\newcommand{\fr}[2]{\frac{#1}{#2}}
	\newcommand{\ms}{\mapsto}

	\newcommand{\ci}{\circ}
	\newcommand{\co}{\colon}
	\renewcommand{\-}{\item}
	\newcommand{\inv}{^{-1}}

	\newcommand{\8}{\infty}
	
	\renewcommand{\.}{\cdot}
	\renewcommand{\:}{\dots}
	\newcommand{\tx}[1]{\text{ #1 }}
	\newcommand{\bl}{{\scriptscriptstyle \bullet}}

\theoremstyle{plain}

\newtheorem{satz}{Satz}[section]		% Im Falle von scrartcl
\newtheorem{theorem}[satz]{Theorem}
\newtheorem{remark}[satz]{Remark}
\newtheorem{lemma}[satz]{Lemma}
\newtheorem{satz/definition}[satz]{Satz/Definition}
\newtheorem{theorem/definition}[satz]{Theorem/Definition}
\newtheorem{definition/theorem}[satz]{Definition/Theorem}
\newtheorem{definition/lemma}[satz]{Definition/Lemma}
\newtheorem{corollary}[satz]{Corollary}
\newtheorem{lemma/definition}[satz]{Lemma/Definition}

\theoremstyle{definition}

\newtheorem{definition}[satz]{Definition}
\newtheorem{convention}[satz]{Convention}

\title{Universal central extensions for groups of sections on non-compact manifolds}
\setdefaultitem{\textbullet}{-}{}{}		% Auflistungensitems 
\setdefaultenum{(a)}{(i)}{}{}			% Aufzaehlungensitems

\author{Jan Milan Eyni}
\date{}
\begin{document}
\maketitle

\begin{abstract}
We construct a central Lie group extension for the Lie group of compactly supported sections of a Lie group bundle over a sigma-compact base manifold. This generalises a result of the paper ``Central extensions of groups of sections'' by Neeb and Wockel, where the base manifold is assumed to be compact. In the second part of the paper, we show that this extension is  universal and obtain a generalisation of a corresponding result in the paper ''Universal central extensions of gauge algebras and groups'' by Janssens and Wockel, where again (in the case of Lie group extensions) the base manifold is assumed compact.
\end{abstract}
{\footnotesize
{\bf Jan Milan Eyni},
Universit\"{a}t Paderborn,
Institut f\"{u}r Mathematik,
Warburger Str.\ 100,
33098 Paderborn, Germany;
\,{\tt janme@math.upb.de}\\[2mm]

\section*{Introduction and notation}
Central extensions play an important role in the theory of infinite-dimensional Lie groups. For example, every Banach-Lie algebra $\fg$ is a central extension $\fz(\fg) \hra \fg \ra \ad(\fg)$, where the centre $\fz(\fg)$ and $\ad(\fg)$ are integrable to a Banach-Lie group; integrability of $\fg$ corresponds to the existence of a corresponding central Lie group extensions (see \cite{van-Est:1964}).

Inspired by the seminal work by van Est and Korthagen, Neeb elaborated the general theory of central extensions of Lie groups that are modelled over locally convex spaces in 2002 (see \cite{Neeb:2002}). In particular, Neeb showed that certain central extensions of Lie algebras can be integrated to central extensions of Lie groups: If the central extension of a locally convex Lie algebra $V\hra \hat{\fg}\ra \fg$ (with a sequentially complete locally convex space $V$) is represented by a continuous Lie algebra cocycle $\o\co \fg^2 \ra V$ and $G$ is a Lie group with Lie algebra $\fg$, one considers the so-called period homomorphism
\begin{align*}
\per_\o \co \p_2(G) \ra V,~ [\s]\ms \int_{\s} \o^l
\end{align*}
where $\o^l \in \go^2(G,V)$ is the canonical left invariant $2$-form on $G$ with $\o^l_1(v,w)=\o(v,w)$ and $\s$ is a smooth representative of the homotopy class $[\s]$. One writes $\gp_\o$ for the image of the period homomorphism and calls it the period group of $\o$. The important result from \cite{Neeb:2002} is that if $\gp_\o$ is a discrete subgroup of $V$ and the adjoined action of $\fg$ on $\hat{\fg}$ integrates to a smooth action of $G$ on $\hat{\fg}$, then $V\hra \hat{\fg}\ra \fg$ integrates to a central  extension  of Lie groups (see \cite[Proposition 7.6 and Theorem 7.12]{Neeb:2002}).

In the following, a Lie group is always assumed to be modelled over a Hausdorff locally convex space.

Given two central Lie group extensions $Z_1\hra \hat{G}_1 \xra{q_1} G$ and $Z_2\hra \hat{G}_2 \xra{q_2} G$, we call a Lie group homomorphism $\ph \co \hat{G}_1 \ra \hat{G}_2$ a morphism of Lie group extensions if $q_1 = q_2 \ci \ph$. In an analogous way one defines a morphism of Lie algebra extensions. In this way one obtains categories of Lie group extensions and Lie algebra extensions, respectively, and an object in these categories is called universal if it is the initial one. In 2002 Neeb showed that under certain conditions a central extension of a Lie group 
is universal in the category of Lie group extensions if its corresponding Lie algebra extension is universal in the category of Lie algebra extensions (see \cite[Recognition Theorem (Theorem 4.13)]{Neeb:2002a}).

The natural next step was to apply the general theory to different types of Lie groups that are modelled over locally convex spaces. Important infinite-dimensional Lie groups are current groups. These are groups of the form $C^\8(M,G)$ where $M$ is a compact finite-dimensional manifold and $G$ is a Lie group. 
In 2003 Maier and Neeb constructed a universal central extensions for current groups (see \cite{Maier:2003}) by reducing the problem to the case of loop groups $C^\8(\mathbb{S}^1,G)$.

The compactness of $M$ is a strong condition but it is not possible to equip $C^\8(M,G)$ with a reasonable Lie group structure if $M$ is non-compact. Although one has a natural Lie group structure on the group $C^\8_c(M,G)$ of compactly supported smooth functions from a $\sigma$-compact manifold $M$ to a Lie group $G$. In this situation, $C^\8_c(M,G)$ is the inductive limit of the Lie groups $C_K^\8(M,G):= \set{f\in C^\8(M,G): \supp(f)\subs K}$ where $K$ runs through a compact exhaustion of $M$. The Lie algebra of $C^\8_c(M,G)$ is given by $C^\8_c(M,\fg)$. In this context, $C^\8_c(M,\fg)$ is equipped with the canonical direct limit topology in the category of locally convex spaces.  
In 2004, Neeb constructed a universal central extension for $C^\8_c(M,G)$ in important cases (see \cite{Neeb:2004}).

It is possible to turn the group $\gg(M,\mathcal{G})$ of sections of a Lie group bundle $\mathcal{G}$ over a compact base manifold $M$ into a Lie group by using the construction of 
the Lie group structure of the gauge group from \cite{Wockel:2007} (see \cite[Appendix A]{Neeb:2009}). The Lie algebra of $\gg(M,\mathcal{G})$ is the Lie algebra $\gg(M,\frak{G})$ of sections of the Lie algebra bundle $\frak{G}$ that corresponds to $\mathcal{G}$. Hence the question arises if it is possible to construct central extensions for these groups of sections. This is indeed the case and was done in 2009 by Neeb and Wockel in \cite{Neeb:2009}.

As mentioned above, one way to show the universality of a Lie group extension is to show the universality of the corresponding Lie algebra extension and then use the Recognition Theorem from \cite{Neeb:2002a}. In the resent paper \cite{Janssens:2013} from 2013, Janssens and Wockel constructed a universal central extension of the Lie algebra $\gg_c(M,\frak{G})$ of compactly supported sections in a Lie algebra bundle over a $\sigma$-compact manifold. They also applied this result to the central extension constructed in \cite{Neeb:2009}: By assuming the base manifold $M$ to be compact they obtained a universal Lie algebra extension that corresponds to the Lie group extension described in \cite{Neeb:2009}; they were able to show the universality of this Lie group extension.

In 2013, Sch{\"u}tt generalised the construction of the Lie group structure from \cite{Wockel:2007} by endowing the gauge group of a principal bundle over a not necessary compact base manifold $M$ with a Lie group structure, under mild hypotheses (see \cite{Schuett:2013}). It is clear that we can use an analogous construction to endow the group of compactly supported sections of a Lie group bundle over a $\sigma$-compact manifold with a Lie group structure. Similarly, Neeb and Wockel already generalised the construction of the Lie group structure on a gauge group with compact base manifold from \cite{Wockel:2007} to the case of section groups over compact base manifolds.

The principal aim of this paper is to construct a central extension of the Lie group of compactly supported smooth sections on a $\sigma$-compact manifold such that its corresponding Lie algebra extension is represented by the Lie algebra cocycle described in \cite{Janssens:2013}.  
This generalises the corresponding result from \cite{Neeb:2009} to the case where the base manifold is non-compact.  
The proof, which combines arguments from \cite{Neeb:2004} and \cite{Neeb:2009} with new ideas, is discussed in Section \ref{GruChLieEx} and Section \ref{Gru1234}. The main result is Theorem \ref{GruMainIII} where we show that the canonical cocycle 
\begin{align*}
\o \co \gg_c(M,\frak{G})^2 \ra \go^1_c(M,\mathbb{V})/d\gg_c(M,\frak{G}),~ (\g,\e)\ms [\k(\g,\e)]
\end{align*}
can be integrated to a cocycle of Lie groups.  This result generalises \cite[Theorem 4.24]{Neeb:2009} to the case of a non-compact base manifold.
 The first step is to show that the period group of  $\o$ is a discrete  subgroup of $\go^1_c(M,\mathbb{V})/d\gg_c(M,\frak{G})$. This will be discussed in Theorem \ref{GruMainI} and is the complementary result to \cite[Theorem 4.14]{Neeb:2009}. Then  we will integrate the adjoint action of $\gg_c(M,\frak{G})$ on $\gg_c(M,\frak{G}) \ti_{\o} \ol{\Omega}^1(M,\mathbb{V})$ to a smooth action of $\gg_c(M,\mathcal{G})$ on $\gg_c(M,\frak{G}) \ti_{\o_M} \ol{\Omega}^1(M,\mathbb{V})$ this is the complementary result to the statements in \cite[Section 4.2 (Part about general Lie algebra bundles]{Neeb:2009}. Considering a compact manifold $M$ in our consideration in  Section \ref{Gru1234}  yields an alternative argumentation for the result in \cite[Section 4.2 (Part about general Lie algebra bundles]{Neeb:2009}\footnote{Our arguments about the discreteness of the image of the period map (Section \ref{GruChLieEx}) dose not  yield an alternative argumentation in the compact case.}. Especially we do not have to assume the typical fiber $G$ of the Lie group bundle $\mathcal{G}$ to be $1$-connected.\footnote{An earlier version of this paper, contained a more complicated argumentation that required the group $G$ to be semisimple.}
In the second part of the paper (Section \ref{GruZweiterTeil}) we turn to the question of universality. Once constructed, the central extension it is not hard to show its universality because mainly we can use the arguments from the compact case (\cite{Janssens:2013}).

In the following we fix our notation:
\begin{compactenum}
\item If $H \hra P \xra{q} M$ is a principal bundle with  right action $R \co P \ti H \ra P$ we write $VP:=\ker(Tq)$ for the vertical bundle of $TP$ and $V_pP:=T_pP\cap VP$ for the vertical space in $p \in P$. Analogously if $HP \subs TP$ is a principal connection ($HP\oplus VP = TP$ and $TR_hH_pP = H_{ph}P$), we write $H_pP:=T_pP \cap HP$ for the horizontal space in $p \in P$. 
\item \label{Gruaaaa}Let $H \hra P \xra{q} M$ be a finite-dimensional principal bundle over a connected $\s$-compact manifold $M$ with  right action $R \co P \ti H \ra P$ and a principal connection $HP \subs TP$. Given a finite-dimensional linear representation $\r \co H \ra \GL(V)$ and $k\in \mathbb{N}_0$, we write 
\begin{align*}
\Omega^k(P,V)_\r = \set{\theta \in \Omega^k(P,V): (\forall g\in H)~\r(g) \ci R_g^\ast\theta = \t} 
\end{align*}
for the space of $H$-invariant $k$-forms on $P$ and $\Omega^k(P,V)_\r^\hor$ for the space of $H$-invariant $k$-forms that are horizontal with respect to $HP$ ($\exists i:~v_i \in V_pP\Rightarrow  \t(v_1,\dots,v_k)=0$) (cf. \cite[Definition 3.3]{Baum:2014}).
Moreover, given a compact set $K \subs M$ we define $\Omega^k_K(P,V)_\r := \set{\t \in \Omega^k_K(P,V)_\r: \supp (\t) \subs q\inv(K)}$ and write $\Omega^k_K(P,V)_\r^\hor$ for the analogous subspace in the horizontal case. We emphasise that these forms are in general not compactly supported in $P$ its self. As mentioned in the introduction we equip these spaces with the natural Fr{\'e}chet-topology and write $\Omega^k_c(P,V)_\r$ respectively $\Omega^k_c(P,V)_\r^\hor$ for the locally convex inductive limit of the spaces $\Omega^k_K(P,V)_\r$ respectively $\Omega^k_K(P,V)_\r^\hor$. This convention also clarifies what we mean by $C^\8(P,V)_\r$ respectively $C_c^\8(P,V)_\r$. 
\item In Lemma \ref{GruRealisation}, we recall that if $\mathbb{V}$ is the vector bundle associated to a principal bundle as in (\ref{Gruaaaa}), then the canonical isomorphism  of chain complexes $\Omega_c^\bl(P,V)_{\r}^\hor \cong \Omega^\bl_c(M,\mathbb{V})$ (see e.g. \cite[Theorem 3.5]{Baum:2014}) induces isomorphisms of locally convex spaces $\Omega_c^k(P,V)_{\r}^\hor \cong \Omega^k_c(M,\mathbb{V})$. 
\- Given a finite-dimensional vector bundle $V\hra \mathbb{V}\xra{q}M$ over a $\s$-compact manifold $M$, a compact set $K\subs M$ and $k \in \mathbb{N}_0$ we write $\go^k_K(M,\mathbb{V})$ for the space of $k$-forms on $M$ with values in the vector bundle $\mathbb{V}$ and support in $K$. Using the identification $\go^k(M,\mathbb{V}) \cong \gg(M,\gl^k T^\ast M \otimes V)$ we give these spaces the locally convex vector topology described in \cite{Gloeckner} and equip $\go_c^k(M,\mathbb{V})$ with the canonical inductive limit topology.  
\- Given a manifold $M$, we write $C^\8_p(\R,M)$ for the set of proper smooth maps from $\R$ to $M$. However, if $F$ is the total space of a fibre bundle $E\hra F\xra{q}M$, then we define $C^\8_p(\R,F):= \set{f\in C^\8(\R,F): q\ci f\in C^\8_p(\R,M)}$. 
\end{compactenum}

\section{Construction of the Lie group extension}\label{GruChLieEx}
We introduce the following conventions:
\begin{convention}\label{GruConvention9898}
\begin{compactenum}
\item All finite-dimensional manifolds are assumed to be $\sigma$-compact.
\item Analogously to \cite[p. 385 and p.388]{Neeb:2009} we consider following setting\footnote{In \cite{Neeb:2009} Neeb and Wockel also consider situations where the Lie groups $H$ and $G$ can be infinite-dimensional locally exponential Lie groups. See also Theorem \ref{GruMainI.2}, where we discuss the infinite-dimensional case.}: If not defined otherwise,  $H\hra P \xra{q}M$ denotes a finite-dimensional principal bundle (see also Theorem \ref{GruMainI.2} for the case where $H$ is infinite-dimensional) over a connected non-compact, $\s$-compact manifold $M$ and $\fh$ the Lie algebra of $H$\footnote{Like in \cite{Neeb:2004} it is crucial for our proof that the manifold $M$ is not compact. Hence our argumentation is not an alternative for the proof of \cite{Neeb:2009}}. Moreover let $G$ be a finite-dimensional Lie group (see also Theorem \ref{GruMainI.2} for the case where $H$ is infinite-dimensional) with Lie algebra $\fg$ and $\k_\fg \co \fg\ti \fg \ra V(\fg)=:V$ be the universal invariant bilinear map on $\fg$ (see e.g.  \cite[Chapter 4]{Gundogan:2011}). Let $\r_G \co H \ti G \ra G$ be a smooth action of $H$ on $G$ by Lie group automorphisms and $\r_\fg \co H \ti \fg \ra \fg$ be the derived action on $\fg$  by Lie algebra automorphisms ($\r_\fg(h,\bl)= L(\r_G(h,\bl)) \in \Aut(\fg)$).
We find a unique map $\r_V \co H \ti V \ra V$ that is linear in the second argument and fulfils $\r_V(h,\k_\fg(x,y)) = \k_\fg(\r_\fg(h,x),\r_\fg(h,y))$ for $x,y\in \fg$ and $h \in H$.
The vector space $V$ is generated by elements of the form $\kappa_\fg(x,y)$ with $x,y\in \fg$. To see that $\rho_V$ is also a representation we show $\rho_V(g,\rho_V(h,\kappa_\fg(x,y))) =  \rho_V(gh,\kappa_\fg(x,y))$ for $x,y\in \fg$ and $g,h \in H$:
\begin{align*}
&\rho_V(g,\rho_V(h,\kappa_\fg(x,y))) = \rho_V(g, \kappa_\fg(\rho_\fg(h,x), \rho_\fg(h,y)))\\
=& \kappa_\fg(\rho_\fg(g,\rho_{\fg}(h,x)), \rho_\fg(g,\rho_{\fg}(h,y)))
= \rho_V(gh,\kappa_\fg(x,y)).
\end{align*}
Because we can find a basis of $V$ consisting of vectors of the form $\kappa_\fg(x,y)$ the smoothness of $\rho_V$ follows.
We write $\mathcal{G}:= P\ti_{\r_G} G$ for the associated Lie group bundle (the definition of a Lie group bundle (respectively associated Lie group bundle) is completely analogous to the definition of a vector bundle (respectively associated Lie group bundle), just in the category of Lie groups), $\frak{G}: = P\ti_{\r_\fg} \fg$ for the associated Lie algebra bundle and $\mathbb{V}: = P\ti_{\r_V} V$ for the associated vector bundle to $H\hra P \ra M$. Let $VP$ be the vertical bundle of $TP$. We  fix a principal connection $HP \subs TP$ on the principal bundle $P$ and write $\pr_h\co TP\ra HP$ for the projection onto the horizontal bundle.
As pointed out in \cite[p. 385]{Neeb:2009} it is no loose of generality to assume the total space $P$ to be connected. Hence we do so in this paper.
\item Let $D_{\r_\fg} \co C^\8_c(P,\fg)_{\r_\fg} \ra \Omega^1_c(P,\fg)_{\r_\fg}^\hor$, $f\ms df\ci\pr_h$ and $D_{\r_V} \co C^\8_c(P,V)_{\r_V} \ra \Omega^1_c(P,V)_{\r_V}^\hor$, $f\ms df\ci\pr_h$ be the absolute derivatives corresponding to $HP$ (cf. \cite[Definition 3.8]{Baum:2014}). Moreover let $d_\frak{G} \co \gg_c(M,\frak{G}) \ra \Omega^1_c(M, \frak{G})$ and $d_\mathbb{V} \co \gg_c(M, \mathbb{V}) \ra \Omega^1_c(M, \mathbb{V})$ be the induced covariant derivations on the Lie algebra bundle $\frak{G}$ and the vector bundle $\mathbb{V}$ respectively (cf. \cite[p. 100 ff]{Baum:2014} and Lemma \ref{GruRealisation}).
\end{compactenum}
\end{convention}

In \cite[Appendix A]{Neeb:2009}, where $M$ is compact, Neeb and Wockel endowed the group of sections $\gg(M,\mathcal{G})$ of a Lie group bundle $\mathcal{G}$ that comes from a principal bundle $P$ with a Lie group structure. They used the identification $\gg(M,\mathcal{G})\cong C^\8(P,G)_{\r_G}$ and endowed the group $C^\8(P,G)_{\r_G}$ of $G$-invariant smooth maps from $P$ to $G$ with a Lie group structure by using the construction of a Lie group structure on the gauge group $\mathrm{Gau}(P)$ described in \cite{Wockel:2007}. To this end they replaced the conjugation of the structure group on itself by the Lie group action $\r_G$. 
In the following Definition \ref{GruLieGruppenStruktur}, we proceed analogously in the case where $M$ is non-compact but $\sigma$-compact. As the construction from \cite{Neeb:2009} is based on \cite{Wockel:2007}, our analogous definition is based on \cite[Chapter 4]{Schuett:2013}, because \cite[Chapter 4]{Schuett:2013} is the generalisation of \cite{Wockel:2007} to the non-compact case.

\begin{definition}\label{GruLieGruppenStruktur}
\begin{compactenum}
\- We equip the group 
\begin{align*}
C^\8_c(P,G)_{\r_G} = \{&\ph\in C^\8(P,G): (\exists K \subs M \text{ compact}) \supp(\ph)\subs q\inv(K)\\
&\tx{and} (\forall h\in H, p \in P) ~ \r_G(h)\ci \ph(pg) =\ph(p)\} 
\end{align*}
with the infinite-dimensional Lie group structure described in \cite[Chapter 4]{Schuett:2013}. We just replace the conjugation of $H$ on itself by the action $\r_G$ of $H$ on $G$.  We emphasise that the functions  $f\in  C^\8_c(P,G)_{\r_G}$ are not compactly supported in $P$ itself. The Lie algebra of $C^\8_c(P,G)_{\r_G}$ is given by the locally convex  Lie algebra
\begin{align*}
&C^\8_c(P,\fg)_{\r_\fg}= \{f\in C^\8(P,\fg)_{\r_\fg}: (\exists K \subs M \text{ compact}) \supp(f)\subs q\inv(K)\}\\% \\ 
=& \varinjlim C^\8_K(P,\fg)_{\r_\fg},
\end{align*}
where $K$ runs through the compact subsets of $M$.
\- From \cite[Chapter 4]{Schuett:2013} (cf. \cite[Theorem 3.5]{Baum:2014} and Lemma \ref{GruRealisation}) we know $\gg_c(M,\frak{G}) \cong C^\8_c(P,\fg)_{\r_\fg}$ in the sense of topological vector spaces. Now, we endow $\gg_c(M,\mathcal{G})$ with the Lie group structure that turns the group isomorphism $\gg_c(M,\mathcal{G})\cong C^\8_c(P,G)_{\r_G}$ into an isomorphism of Lie groups. Hence $\gg_c(M,\mathcal{G})$ becomes an infinite-dimensional Lie group modelled over the locally convex space $\gg_c(M,\frak{G})$.
\end{compactenum}
\end{definition}

In the following definition we fix our notation for the quotient principal bundle. For details on the well-known concept of quotient principal bundles see e.g. \cite[Proposition 2.2.20]{Gundogan:2011}. 
\begin{definition}\label{GruQuotientenBuendel}
Let $N := \ker(\r_V) \subs H$ and  $H/N \hra P/N \xra{\ol{q}} M$ be the quotient bundle with projection $\ol{q} \co P/N \ra M,~ pN \ms q(p)$ and right action $\ol{R} \co H/N \ti P/N \ra P/N,~ ([g],pN) \ms (pg)N$. We write $\ol{H}:=H/N$ and  $\ol{P}:=P/N$. Let $\ol{\r}_V\co H/N \ra GL(V)$ be the factorisation of $\r_V$ over $N$ and $\p \co P \ra P/N$ the orbit projection. If $\ps\co q\inv(U)\ra U\ti H$ is a trivialisation of $P$, then $\ps'\co \ol{q}\inv(U)\ra U\ti \ol{H}$, $pN \ms (q(p), [\pr_2\ci \ps(p)])$ is a typical trivialisation of $\ol{P}$.  
It is well-known that $\mathbb{V}$ is isomorph to the associated bundle to $\ol{H} \hra \ol{P}\xra{\ol{q}} M$ via $\ol{\r}_V$ (see e.g. \cite[Remark 2.2.21]{Gundogan:2011}).
Moreover, we write $H\overline{P}:= T\pi(HP)$ for the canonical principal connection on $\overline{P}$ that comes from $P$ (see \ref{GruIsoQuotientBundle} (a)). We mention that $\pi^\ast \colon \Omega^k_c(\overline{P},V)^\hor_{\overline{\rho}} \rightarrow \Omega_c^k(P,V)^\hor_{\rho_V}$ is an isomorphism of topological vector spaces and induces an isomorphism of chain complexes (see \ref{GruIsoQuotientBundle} (c)).
\end{definition}

\begin{convention}
Analogously to \cite{Neeb:2009} we introduce the following convention. We assume that the identity-neighbourhood of $H$ acts trivially on $V$ by $\r_V$ (cf. \cite[p. 385]{Neeb:2009}). Hence $\ol{H}$ is a discrete Lie group. Moreover, we even assume $\ol{H}$ to be finite (cf. \cite[p. 386, p.398 f and Theorem 4.14]{Neeb:2009}). 
\end{convention}

\begin{definition}
Let $H\hra P \ra M$ be a principal bundle with connected total space $P$ and $\r_V \co H \ti V \ra V$ a linear representation. Moreover fix a connection $HP$ on $TP$ and let $D_{\r_V}$ be the induced absolute derivative of the associated vector bundle $\mathbb{V}$.
\begin{compactenum}
\item We define
\begin{align*}
&Z^1_{dR,c}(P,V)_{\r_V}:= \set{\t \in \Omega_c^1(P,V)^\hor_{\r_V}: D_{\r_V}\t =0}, \\
&B^1_{dR,c}(P,V)_{\r_V}:= D_{\r_V}(C^\8_c(P,V)_{\r_V}),\\
\end{align*}
and equip these spaces with the induced topology of $\Omega_c^1(P,V)^\hor_{\r_V}$.
\item We define
\begin{align*}
&Z^1_{dR,c}(P,V)_{\fix}:= Z^1_{dR,c}(P,V) \cap \Omega^1_c(P,V)_{\r_V},\\ 
&B^1_{dR,c}(P,V)_{\fix}:= B^1_{dR,c}(P,V) \cap \Omega_c^1(P,V)_{\r_V}
\end{align*}
and equip these spaces with the induced topology of $\Omega_c^1(P,V)_{\r_V}$.
\end{compactenum}
\end{definition}

\begin{lemma}\label{GruFIX88}
Let $H\hra P \ra M$ be a principal bundle and $\r_V \co H \ti V \ra V$ a linear representation. Moreover fix a connection $HP$ on $TP$ and let $D_{\r_V}$ be the induced absolute derivative of the associated vector bundle $\mathbb{V}$.
\begin{compactenum}
\item If $H$ is discrete we have
\begin{align*}
Z^1_{dR,c}(P,V)_{\r_V} = Z^1_{dR,c}(P,V)_\fix \tx{and} B^1_{dR,c}(P,V)_{\r_V} \subs B^1_{dR,c}(P,V)_\fix.
\end{align*}
Because in this situation all forms on $P$ are horizontal the topologies  on $Z^1_{dR,c}(P,V)_{\r_V}$ and  $Z^1_{dR,c}(P,V)_\fix$ coincide.
\- \label{GruFIXB1} If $H$ is finite we get $B^1_{dR,c}(P,V)_{\r_V} = B^1_{dR,c}(P,V)_\fix$. Again the topologies on these subspaces coincide, because the $\Omega^1_{dR,c}(P,V)_{\r_V}^\hor$ and $\Omega^1_{dR,c}(P,V)_{\r_V}$ are exactly the same topological vector spaces.
\end{compactenum}
\end{lemma}
\begin{proof}
\begin{compactenum}
\item If $H$ is discrete there is only one connection on $P$ namely $HP=TP$. Hence $D_{\r_V}$ becomes the normal exterior derivative.
\- Let $n:=\#H$ and $\t \in B^1_{dR,c}(P,V)_\fix$ with $\t = df$ for $f\in C^\8_c(P,V)$. For $\ph \in C^\8_c(P,V)$ and $g\in H$ we write $g.\ph := \r_V(g)\ci R_g^\ast \ph$ and get $\fr{1}{n} \. \sum_{g\in H} g.f \in C^\8_c(P,V)_{\r_{V}}$. Moreover $d(\fr{1}{n} \. \sum_{g\in H} g.f) = \t$. Hence $B^1_{dR,c}(P,V)_{\r_V} = B^1_{dR,c}(P,V)_\fix$.
\end{compactenum}
\end{proof}

\begin{lemma}\label{GruBkdlc89}
Let $H\hra P \xra{q} M$ be a principal bundle with finite structure group $H$ and connected total space $P$. Moreover let $\r_V \co H \ti V \ra V$ be a finite-dimensional linear representation, $HP$ a connection on $TP$ and $D_{\r_V}$ be the induced absolute derivative of the associated vector bundle $\mathbb{V}$.
\begin{compactenum}
\item  The map $q$ is proper. Hence in this case the forms in $\Omega^k_c(P,V)$ are exactly the compactly supported forms in $P$.
\item The space $B^1_{dR,c}(P,V) = dC^\8_c(P,V)$ is a closed subspace of $\Omega^1_c(P,V)$.
\end{compactenum}
\end{lemma}
\begin{proof}
\begin{compactenum}
\item \cite[Lemma 10.2.11]{Napier:2011} tells us that  if $F\hra \mathbb{F} \xra{q} M$ is a continuous fibre bundle of finite-dimensional topological manifolds and $F$ is finite, then $q$ is a proper map (a more general statement in the setting of topological spaces is stated in \cite[Exercise A.75.]{Lee:2013} (but it is not of interest for our considerations)).
\item \cite[Lemma IV.11]{Neeb:2004} tells us that, if $M$ is a connected finite-dimensional manifold and $V$ a finite-dimensional vector space, then $B^1_{dR,c}(M,V)=dC^\8_c(M,V)$ is a closed subspace of $\Omega^1_c(M,V)$. 
\end{compactenum}
\end{proof}

For the corresponding statement to the following lemma in the case of a compact base manifold, compare \cite[p. 385 f]{Neeb:2009}.
\begin{lemma}\label{GruAbgeschlossenDrVff}
The subspace $D_{\r_V}C^\8_c(P,V)_{\r_V} \subs \Omega_c^1(P,V)_{\r_V}^\hor$ is closed. 
\end{lemma}
\begin{proof}
The lemma simply says that $d\gg_c(M,\mathbb{V})$ is closed in $\Omega_c^1(M,\mathbb{V})$. Hence it is enough to show that the subspace $dC^\8_c(\ol{P},V)_{\ol{\r}_V}$ is closed in $\Omega_c^1(\ol{P},V)_{\ol{\r}_V}^\hor = \Omega_c^1(\ol{P},V)_{\ol{\r}_V}$. We know  that $B^1_{dR,c}(\ol{P},V)$ is closed in $\Omega^1_c(\ol{P},V)$. We calculate 
\begin{align*}
&dC^\8_c(\ol{P},V)_{\ol{\r}_V} = B^1_{dR,c}(\ol{P},V)_{\fix} = \bigcap_{g\in \ol{H}} \set{\t \in B^1_{dR,c}(\ol{P},V): \ol{\r}_V(g)\ci \ol{R}_g^\ast \t =\t}\\
= &\bigcap_{g\in \ol{H}} (\ol{\r}_V(g)\ci \ol{R}_g^\ast - \id)\inv\set{0}.
\end{align*}
We see that $dC^\8_c(\ol{P},V)_{\ol{\r}_V}$ is closed in $\Omega^1_c(\ol{P},V)$. Because the topology of $\Omega^1_c(\ol{P},V)_{\rho_V}$ is finer then the induced topology of $\Omega^1_c(\ol{P},V)$, the space $dC^\8_c(\ol{P},V)_{\ol{\r}_V}$ is also closed in $\Omega^1_c(\ol{P},V)_{\rho_V}$.
\end{proof}

\begin{definition}
Let $H\hra P \ra M$ be a principal bundle with connected total space $P$ and $\r_V \co H \ti V \ra V$ a linear representation. Moreover fix a connection $HP$ on $TP$ and let $D_{\r_V}$ be the induced absolute derivative of the associated vector bundle $\mathbb{V}$.
\begin{compactenum}
\item If the quotient group $H/\ker(\r_V)$ is finite (this of course includes the case where the group $H$ is finite) we define
\begin{align*}
H^1_{dR,c}(P,V)_{\r_V}:=Z^1_{dR,c}(P,V)_{\r_V} / B^1_{dR,c}(P,V)_{\r_V} \cong H^1_{dR,c}(M,\mathbb{V}).
\end{align*}
Because of Lemma \ref{GruAbgeschlossenDrVff} this is a Hausdorff locally convex space.
\item We have a canonical $H$-module structure on $H^1_{dR,c}(P,V)$ given by $H\ti H^1_{dR,c}(P,V) \ra H^1_{dR,c}(P,V)$, $(h,[\t]) \ms [\r_V(h)\ci R_h^\ast \t]$. As usual we call the fixed points of this action $\r_V$-invariant. If the group $H$ is finite we define
\begin{align*}
H^1_{dR,c}(P,V)_{\fix}:= \set{ [\t] \in H^1_{dR,c}(P,V): [\t]\tx{is}\r_V\text{-invariant}}
\end{align*}
and because of Lemma \ref{GruBkdlc89} the space  $H^1_{dR,c}(P,V)_{\fix}$ becomes a Hausdorff locally convex space as a closed subspace of the Hausdorff locally convex space $H^1_{dR,c}(P,V)$.
\end{compactenum}
\end{definition}

It is possible to show the following Lemma \ref{GruFIX}
by a more abstract argument using that under certain conditions the fixed point functor is exact like it was done in the compact case in  \cite[Remark 4.12]{Neeb:2009}.
\begin{lemma}\label{GruFIX}
Let $H\hra P \ra M$ be a principal bundle and $\r_V \co H \ti V \ra V$ a linear representation. 
If $H$ is finite we get
\begin{align*}
H^1_{dR,c}(P,V)_\fix  \cong Z^1_{dR,c}(P,V)_{\fix} / B^1_{dR,c}(P,V)_{\fix}.
\end{align*}
in the sense of topological vector spaces.
\end{lemma}
\begin{proof}
Let $n:=\#H$. We consider the linear map $\ps \co Z^1_{dR,c}(P,V)_\fix \ra H^1_{dR,c}(P,V)_\fix,~ \t \ms [\t]$. The map $\ps$ is continuous because the inclusion $Z_{dR,c}(P,V)_\fix \hra \Omega^1_{c}(P,V)$ is continuous and so the canonical map $Z_{dR,c}(P,V)_\fix \ra H^1_{dR,c}(P,V)$ is continuous. If $[\t] \in H^1_{dR,c}(P,V)_\fix$ with $\t = df$ for $f\in C^\8_c(P,V)$, then $[\t] = [d (\fr{1}{n} \sum_{g\in H}g.f)]$ and  $d (\fr{1}{n} \sum_{g\in H}g.f) \in B^1_{dR,c}(P,V)_\fix$ so $\ker(\ps) \subs B^1_{dR,c}(P,V)_\fix$. Obviously $B^1_{dR,c}(P,V)_\fix \subs \ker(\ps)$. Now we show that $\ps$ is surjective. If $[\t] \in H^1_{dR,c}(P,V)_\fix$, then $[\t] = [\fr{1}{n}\. \sum_{g\in H} g.\t]$ and $\fr{1}{n}\. \sum_{g\in H} g.\t \in Z^1_{dR,c}(P,V)_\fix$. Hence $\ps$ factors through a continuous bijective linear map $\ol{\ps} \co  Z^1_{dR,c}(P,V)_{\fix} / B^1_{dR,c}(P,V)_{\fix} \ra H^1_{dR,c}(P,V)_\fix$. It is left to show that $\ps$ is also open. We define 
\begin{align*}
\tau \co H_{dR,c}^1(P,V)\ra Z^1_{dR,c}(P,V)_{\fix} / B^1_{dR,c}(P,V)_{\fix}, [\t] \ms \left[\frac{1}{n} \cdot \sum_{g\in H} g.\t\right]. 
\end{align*}
Obviously $\tau|_{H_{dR,c}^1(P,V)_\fix}$ is inverse to $\ol{\ps}$. The map  
\begin{align*}
\Omega^1_{c}(P,V) \ra \Omega^1_{c}(P,V)_{\r_V},~\t \ms \frac{1}{n} \cdot \sum_{g\in H} g.\t
\end{align*}
is continuous, because the action $g.\t = \rho(g)\ci R_g^\ast \t$ dose not enlarge the support of a given form.
\end{proof}

\begin{corollary}\label{GruCorolllaryol}
Considering the principal bundle $\ol{H}\hra \ol{P} \ra M$ with the action $\ol{\r}_V$ we have 
\begin{align*}
H^1_{dR,c}(\ol{P},V)_\fix  \cong Z^1_{dR,c}(\ol{P},V)_{\fix} / B^1_{dR,c}(\ol{P},V)_{\fix}\cong H^1_{dR,c}(\ol{P},V)_{\ol{\r}_V}.
\end{align*}
\end{corollary}

The following lemma is a generalisation of considerations in \cite[p. 399 and Remark 4.12]{Neeb:2009} from the compact case to the non-compact case.
\begin{lemma}
\begin{compactenum}
\item If we endow $H^1_{dR,c}(M,V)$ with the canonical $\ol{H}$-module structure $\ol{H}\ti H^1_{dR,c}(M,V) \ra H^1_{dR,c}(M,V),~ (h,[\t]) \ms [\ol{\r}_V(h) \ci \t]$, the map $\ol{q}^\ast \co H^1_{dR,c}(M,V) \ra H^1_{dR,c}(\ol{P},V)$ becomes an isomorphism of $\ol{H}$-modules. And $H^1_{dR,c}(M,V)_\fix \cong_{\ol{q}^\ast} H^1_{dR,c}(\ol{P},V)_\fix$.
\item We have 
\begin{align}\label{GruIsoDeRham}
H^1_{dR,c}(M,V_\fix) \cong H^1_{dR,c}(M,V)_\fix
\end{align}
if we write $V_\fix$ for the sup space of fixed points in $V$ by the action $\ol{\r}_V$.  
\item 
The map 
\begin{align*}
H^1_{dR,c}(M,V_\fix) \ra H_{dR,c}^1(P,V)_{\r_V},~ [\t] \ms [q^\ast \t]
\end{align*}
is an isomorphism of topological vector spaces.
\end{compactenum}
\end{lemma}
\begin{proof}
\begin{compactenum}
\item For $\ol{h}\in \ol{H}$ we calculate
\begin{align*}
\ol{q}^\ast [\ol{\r}(\ol{h}) \ci \t] = [\ol{\r}(\ol{h}) \ci \ol{q}^\ast \t]  = [\ol{\r}(\ol{h}) \ci (\ol{q}\ci \ol{R}_{\ol{h}})^\ast \t] = [\ol{\r}(\ol{h}) \ci \ol{R}_{\ol{h}}^\ast \ol{q}^\ast \t] = \ol{h}.\ol{q}^\ast [\t].
\end{align*}
Hence $\ol{q}^\ast$ is an isomorphism of $\ol{H}$-modules. Now the second assertion follows from Lemma \ref{GruIsoEndlCover}.
\item We exchange $P$ with $M$ and the action $g.\t= \ph(g)\ci R_g^\ast\ph$ with $g.\t= \ol{\rho}_V(g)\ci \t$ in the proof of Lemma \ref{GruFIX} 
and get
\begin{align*}
H^1_{dR,c}(M,V)_\fix \cong Z^1_{dR,c}(M,V)_\fix /B^1_{dR,c}(M,V)_\fix.
\end{align*}
Now we show that the isomorphism $\ph \co \Omega^1_{c}(M,V_\fix)\ra \Omega_{c}^1(M,V)_\fix$, $\t\ms\t$ is a homeomorphism where $\Omega_{c}^1(M,V)_\fix$ is equipped with the induced topology from $\Omega_{c}^1(M,V)$. Given a compact set $K\subs M$ the map $\Omega^1_{K}(M,V_\fix)\ra \Omega_{K}^1(M,V)$ is continuous. Hence $\Omega^1_{c}(M,V_\fix)\ra \Omega_{c}^1(M,V)$ is continuous. Therefore $\ph$ is continuous. Considering the continuous map $\Omega^1_{c}(M,V)\ra \Omega^1_{c}(M,V_\fix)$, $\t \ms \sum_{\ol{h} \in \ol{H}} \ol{h}.\t$, we see that $\ph$ is an isomorphism of topological vector spaces.
Now the assertion follows from $Z^1_{dR,c}(M,V)_\fix=Z^1_{dR,c}(M,V_\fix)$ and $B^1_{dR,c}(M,V)_\fix=B^1_{dR,c}(M,V_\fix)$.
\item We have the commutative diagram
\begin{align*}
\begin{xy}\xymatrixcolsep{5pc}
\xymatrix{
H_{dR,c}^1(M,V_\fix) \ar[rr]^-{q^\ast} \ar[d]&& H_{dR,c}^1(P,V)_{\rho_V}\\
H_{dR,c}^1(M,V)_\fix \ar[r]^-{\ol{q}^\ast} & H^1_{dR,c}(\ol{P},V)_\fix \ar[r] &  H^1_{dR,c}(\ol{P},V)_{\ol{\r}_V} \ar[u]^-{\p^\ast}.
}
\end{xy}
\end{align*}
The assertion now follows from (a), (b) and Corollary \ref{GruCorolllaryol}.
\end{compactenum}
\end{proof}

\begin{convention}
From now on we write $q_\ast \co H_{dR,c}^1(P,V)_{\r_V} \ra H^1_{dR,c}(M,V_\fix)$ for the inverse of $q^\ast \co H^1_{dR,c}(M,V_\fix) \ra H_{dR,c}^1(P,V)_{\r_V},~ [\t] \ms [q^\ast \t]$.
\end{convention}

\begin{remark}\label{GruVollststaendigno}
Given an infinite-dimensional Lie group $G$ with Lie algebra $\fg$, a trivial locally convex $\fg$-module $\frak{z}$ and a Lie algebra cocycle $\o\co \fg\ti \fg \ra \fz$, \cite[Theorem 7.12]{Neeb:2002} gives us conditions under which we can integrate $\o$ to a Lie group cocycle of the Lie group $G$. These conditions were recalled in the introductions to this thesis (see p. xiv). Theorem 7.12 in \cite{Neeb:2002} is formulated in the case where $\fz$ is sequentially complete\footnote{See Lemma \ref{GruFolgenvoll2} for conditions that guaranty the sequential completeness of $\ol{\Omega}^1(M,\mathbb{V})$.}. Although it also hold in a special case when $\fz$ is not sequentially complete: Let $E$ be a Mackey complete space, $F\subs E$ be a closed subspace, $\fz= E/F$. If $\o$ lifts to a continuous bilinear map $\a \co \fg\ti \fg \ra E$, then the results of \cite{Neeb:2002} stay valid. To see this we make the following consideration: Let $\o^l$ be the left invariant $2$ form on $G$ corresponding to $\o$. 
The completeness of $\fz$ is only used to guaranty the existence of weak integrals in the following settings: 
\begin{compactenum}
\item $\int_{{\s}}\o^l=\int_M\s^\ast \o^l$ where $M$ is a $2$-dimensional manifold (namely $M=\mathbb{S}^1$) or simplex and $\s\co M\ra G$ is a smooth map (see \cite[Chapter 5 and 6]{Neeb:2002}),
\item $\int_0^1\o^l(f(t))dt$ where $f \co [0,1]\ra TG\oplus TG$ is a smooth map into the Whitney sum (see \cite[Chapter 7]{Neeb:2002}).
\end{compactenum}
The integrals $\int_{{\s}}\o^l$ and $\int_0^1\o^l(f(t))dt$ are weak integral, but such integrals do not have to exist in arbitrary locally convex space. Although they exist in sequentially complete (respectively Mackey complete) locally convex spaces. This is the reason why Neeb assumes $\fz$ to be sequentially complete. 
Now we consider the situation where $\fz$ is not its self sequentially complete, but $\fz = E/F$ with a Mackey complete locally convex space $E$ and a closed subspace $F$ and $\o = \p \ci \a$ is a Lie algebra cocycle with the canonical projection $\p \co E \ra E/F$ and a continuous bilinear map $\a \co \fg^2\ra E$. We  show the existence of the weak integral $\int_{\s}\o^l$. 
We define $\til{\a}\co \fg^2 \ra E,~ (v,w) \ms \fr{1}{2} \a(v,w) - \fr{1}{2} \a(w,v)$ (cf. \cite[Remark 2.2]{Neeb:2009}). We get $\p \ci \til{\a} = \o$ and $\til{\a}$ is a continuous Lie algebra 2-cochain. Let $\til{\a}^l \in \Omega^2(G,E)$ be the left invariant differential form on $G$ that comes from $\til{\a}$. We get $\o^l = \p \ci \til{\a}^l$ and the weak integral $\int_\s \o^l$ is given by
\begin{align*}
\int_{M} \s^\ast \o^l = \p \left( \int_{M}\s^\ast \til{\a}^l\right).
\end{align*}
The existence of the weak integral $\int_0^1\o^l(f(t))dt$ follows analogously.
\end{remark}

\begin{definition}\label{Gruomega}
We define the locally convex spaces
\begin{align*}
&\ol{\Omega}^1_c(P,V)_{\r_V}^\hor := \Omega^1_c(P,V)_{\r_V}^\hor / D_{\r_V}C^\8_c(P,V)_{\r_V} \tx{and}\\
&\ol{\Omega}^1_c(M,\mathbb{V}) := \Omega^1_c(M,\mathbb{V}) / d_\mathbb{V}\gg_c(M,\mathbb{V}).
\end{align*}
With Lemma \ref{GruIsoQuotientBundle} and Lemma \ref{GruRealisation} we get
\begin{align*}
\ol{\Omega}^1_c(M,\mathbb{V}) \cong \ol{\Omega}^1_c(P,V)_{\r_V}^\hor \cong \ol{\Omega}^1_c(\ol{P},V)_{\ol{\r}_V}^\hor. 
\end{align*}
\end{definition}

\begin{remark}\label{GruRmarkallslsgfla}
\begin{compactenum}
\item Considering the vector bundle $V(\frak{G})$ from \cite{Janssens:2013}, we have a vector bundle isomorphism $\mathbb{V} \ra V(\frak{G})$ given by
\begin{align*}
\ph\co P\ti_{\r_V} V = \mathbb{V}& \ra V(\frak{G}) = V(P \ti_{\r_\fg} \fg),\\
 [p,\k_\fg(x,y)] &\ms \k_{\frak{G}_{q(p)}}([p,x],[p,y]) \tx{for}x,y \in\fg.
\end{align*}
In fact, $\ph$ is well-defined, because given $p\in P$ there exists a unique linear map $\ph_p \co V=V(\fg) \ra V(\frak{G}_{q(p)})$ given by $\ph_p(\kappa_\fg(x,y)) = \kappa_{\frak{G}_{q(p)}}([p,x],[p,y])$. And given $x\in M$ the map $(P\ti_{\rho_V}V)_x \ra V(\frak{G}_x)$, $[p,v] \ms \ph_p(v)$ is well-defined. The bundle morphism $\ph$ is smooth, because $\ph$ is locally given by $U\ti V \ra U\ti V$, $(x_0,\kappa_\fg(x,y)) \ms (x_0,\kappa_\fg(x,y)$ for a domain $U \subs M$ of a trivialisation of $P$, $x_0\in U$ and $x,y \in \fg$ in the canonical charts. Hence $\ph$ is locally given by the identity $U\ti V \ra U\ti V$.
\item \label{GruRmarkallslsgflab} Given $\t \in \Omega^1_c(P,\fg)^\hor_{\rho_\fg}$ and $f \in C_c^\8(P,\fg)_{\rho_\fg}$, we have $\kappa_\fg\ci (\t,f) \in \Omega^1_c(P,V)_{\rho_V}^\hor$. In fact obviously $\kappa_\fg\ci (\t,f)$ is horizontal and ``compactly supported'' with respect to the principal bundle $P \xra{q}M$. Moreover given $h\in H$, $p \in P$ and $v \in T_pP$, we calculate
\begin{align*}
&R_h^\ast (\kappa_\fg\ci (\t,f))_p(v) = \kappa_\fg( \t_{ph}(TR_h(v)), f(ph))\\
=&\kappa_\fg(\rho_\fg(h\inv). \t_p(v), \rho_\fg(h\inv). f(p) )  = \rho_V(h\inv).\kappa_\fg(\t_p(v), f(p)).
\end{align*}
Therefore the map $\til{\kappa}_\fg \co \Omega^1_c(P,\fg)_{\rho_\fg}^\hor \ti C_c^\8(P,\fg)_{\rho_\fg} \ra \Omega^1_c(P,V)_{\rho_V}^\hor$, $(\t,f) \ms \kappa_\fg\ci (\t,f)$ makes sense and we obtain the commutative diagram
\begin{align*}
\begin{xy}\xymatrixcolsep{5pc}
\xymatrix{
\Omega^1_c(P,\fg)_{\rho_\fg}^\hor \ti C_c^\8(P,\fg)_{\rho_\fg}  \ar[dd]  \ar[r]^-{\til{\kappa}_\fg}& \Omega^1_c(P,V)_{\rho_V}^{\mathrm{hor}} \ar[d]\\
& \Omega_c(M,\mathbb{V}) \ar[d]\\
\Omega^1_c(M,\frak{G}) \ti \Gamma_c(\frak{G}) \ar[r]^-{\til{\kappa}_{\frak{G}}} & \Omega^1_c(M,V(\frak{G})),
}
\end{xy}
\end{align*}
where the lower horizontal arrow is given by the map $\til{\kappa}_{\frak{G}}$ described in  \cite[Lemma 3.23]{Eyni:2014} and the vertical arrows are the canonical isomorphisms of topological vector spaces. Especially $\til{\kappa}_\fg$ is continuous. Moreover we write $\til{\kappa}_\fg(\eta,\t):=\til{\kappa}_\fg(\t,\eta)$ for $\t \in \Omega^1_c(P,\fg)_{\rho_\fg}^\hor$ and $\eta\in \Gamma_c(\frak{G})$.
\item The map $C^\8_c(P,\fg)_{\rho_\fg} \ti \Omega^1_c(P,\fg)_{\rho_\fg}^\hor \ra \Omega^1_c(P,\fg)_{\rho_\fg}^\hor$, $(\eta, \t)\ms [\eta, \t]$ with $[\eta, \t]_{p}(w)= [\eta(p), \t_p(w)]$ for $p \in P$ and $w\in T_pP$ makes sense, because obviously $[\eta,\t]$ is horizontal and $[\eta,\t] \in \Omega^1_c(P,\fg)$ and because $\rho_\fg$ acts by Lie algebra automorphisms on $\fg$ the form $[\eta,\t]$ is also $\rho_\fg$ invariant. Under the canonical isomorphisms of topological vector spaces $\Omega^1_c(P,\fg)_{\rho_\fg}^\hor \cong \Omega^1_c(M,\frak{G})$ and $C^\8_c(P,\fg)_{\rho_\fg} \cong \Gamma_c(M,\frak{G})$ this map corresponds to the map $\Omega^1_c(M,\frak{G}) \ti \Gamma_c(M,\frak{G}) \ra \Omega_c(M,\frak{G})$, $(\t,\eta)\ms [\t,\eta]$ with $[\t,\eta]_x(v) = [\t_x(v), \eta(x)]_{\frak{G}_x}$ for $x\in M$ and $v\in T_xM$. We define $[\eta,\t]:=-[\t,\eta]$ for  $\t \in \Omega^1_c(P,\fg)_{\rho_\fg}^\hor$ and $\eta\in \Gamma_c(\frak{G})$.
\item  We write $\pr_h\co TP \ra HP$ for the projection onto the horizontal bundle. We see directly, that $D_{\rho_\fg} \co C^\8_c(P,\fg)_{\rho_\fg} \ra \Omega^1_c(P,\fg)_{\rho_\fg}^\hor$, $f\ms df\ci \pr_h$ is a Lie connection.  
\item We define the map $\beta \co C_c^\8(P,\fg)_{\rho_\fg}\ti C_c^\8(P,\fg)_{\rho_\fg} \ra \Omega_c^1(P,V)_{\rho_V}^\hor$, $\beta(f,g) = \til{\kappa}_\fg(D_{\rho_\fg}f , g) + \til{\kappa}_\fg(D_{\rho_\fg} g, f)$. Because $D_{\rho_V}$ and $D_{\rho_\fg}$ are induced by the same principal connection on $P$ we obtain a commutative diagram
\begin{align*}
\begin{xy}\xymatrixcolsep{5pc}
\xymatrix{
C^\8_c(P,\fg)_{\rho_\fg} \ti C^\8_c(P,\fg)_{\rho_\fg} \ar[r]^-{\beta} \ar[d]^-{(\kappa_\fg)_\ast}& \Omega^1_c(P,V)_{\rho_V}^\hor\\
C^\8_c(P,V)_{\rho_V} \ar[ur]_-{D_{\rho_V}} &
}
\end{xy}
\end{align*}
where $(\kappa_\fg)_\ast(f,g) = \kappa_\fg\ci (f,g)$.
\end{compactenum}
\end{remark}

\begin{definition}\label{GruDeffffomegaMMM}
We define the map
\begin{align*}
\o_M \co C^\8_c(P,\fg)_{\r_{\fg}} \ti C^\8_c(P,\fg)_{\r_{\fg}} \ra \ol{\Omega}_c^1(P,V)_{\r_V}^\hor,~ 
(f,g)\ms [\k_\fg(f,D_{\r_\fg}g)],
\end{align*}
which is the analogous map to the cocycle $\o$ defined in the compact case in \cite[Proposition 2.1]{Neeb:2009}. Because $D_{\r_\fg}$ is linear, $D_{\r_\fg}(C_K^\8(P,\fg)_{\rho_\fg}) \subs \Omega^1_K(P,\fg)_{\rho_\fg}^\hor$ and $D_{\r_\fg}(f) = df\ci \pr_h$ we get that $D_{\r_\fg} \co C^\8_c(P,\fg)_{\rho_\fg} \ra \Omega^1_c(P,\fg)_{\fg}^\hor$ is continuous. Considering Remark \ref{GruRmarkallslsgfla} (\ref{GruRmarkallslsgflab}) we see that $\o_M$ is continuous. Repeating the argumentation of the proof of \cite[Proposition 2.1]{Neeb:2009}  
we see that $\o_M$ is anti-symmetric and a cocycle.
\end{definition}

\begin{remark}\label{Gru89Remarkdkdf}
In this remark we suppose that $\frak{g}$ is perfect. Because $(\kappa_\fg)_\ast \co C^\8_c(P,\fg)_{\rho_\fg} \ti C^\8_c(P,\fg)_{\rho_\fg} \ra C^\8_c(P,V)_{\rho_V}$ corresponds to the universal continuous invariant bilinear form $\kappa_{\frak{G}} \co \Gamma_c(\frak{G}) \ti \Gamma_c(\frak{G}) \ra \Gamma_c(V(\frak{G}))$  from \cite[Theorem 3.20]{Eyni:2014} the absolute derivative $D_{\rho_V}$ corresponds to the covariant derivative $d$ constructed in \cite[Theorem 3.24]{Eyni:2014}, especially we have $d=d_{\mathbb{V}}$.  Hence our Lie algebra cocycle  $\o_M$ from Definition \ref{GruDeffffomegaMMM} corresponds to the cocycle $\o_\nabla$ from \cite[Chapter 1, (1.1)]{Janssens:2013}. 
\end{remark}

In \cite{Neeb:2009} Neeb and Wockel used Lie group homomorphisms that are pull-backs by horizontal lifts of smooth loops $\a \co \mathbb{S}^1 \ra M$ to reduce the proof of the discreteness of the period group to the case of $M=\mathbb{S}^1$ (see \cite[Definition 4.2 and Remark 4.3]{Neeb:2009}). But this approach dose not work in the non-compact case. Instead we want to use the results from \cite{Neeb:2004} of current groups on non-compact manifolds. Hence we use pull-backs by horizontal lifts of proper maps $\a \co \R \ra M$ (see the next definition). 
A corresponding definition to the following Definition \ref{GruDerPullBack}, in the case of a compact base manifold is given by \cite[Definition 4.2]{Neeb:2009}.

\begin{definition}\label{GruDerPullBack}
We fix $x_0 \in M$, $p_0 \in P_{x_0}$ and $\a \in C^\8_p(\R, M)$ with $\a(0)=x_0$. Let $\hat\a\in C^\8(\R,P)$ be the unique horizontal Lift of $\a$ with $\hat\a(0) = p_0$. We define the group homomorphism
\begin{align*}
\hat{\a}_G^\ast \co C^\8_c(P,G)_{\r_G} \ra C^\8_c(\R, G),~ \ph\ms \ph\ci \hat{\a}
\end{align*}
and the Lie algebra homomorphism 
\begin{align*}
\hat{\a}_\fg^\ast \co C^\8_c(P,\fg)_{\r_\fg} \ra C^\8_c(\R, \fg),~ f\ms f\ci \hat{\a}.
\end{align*}
In this context the maps in  $C^\8_c(\R, G)$ respectively $C^\8_c(\R, \fg)$ are compactly supported in $\R$ itself. 
These maps make sense, because given $\ph\in C^\8_c(P,G)$ we have $\supp (\ph) \subs q\inv(L)$ for a compact set $L \subs M$. We get $\supp(\ph\ci \hat{\a}) \subs \a\inv(L)$, because if $\ph(\hat{\a}(t)) \neq 1$, we get $\hat{\a}(t) \in q\inv(L)$ and so $\a(t)= q\ci \hat{\a}(t) \in L$. Hence $t\in \a\inv(L)$. Now we take the closure.
Moreover we define the integration map
\begin{align*}
I_\a \co \ol{\Omega}_c^1(P,V)_{\r_V}^\hor \ra V,~ [\t] \ms \int_{\R}\hat{\a}^\ast \t.
\end{align*}
This map is  well-defined: Let $\t \in \Omega^1_c(P,V)_{\r_V}$ with $\supp(\t) \subs q\inv(L)$ for a compact set $L \subs M$. We get $\supp(\hat{\a}^\ast \t) \subs \a\inv(L)$, because if $(\hat{\a}^\ast \t)_t \neq 0$ we get $\hat{\a}(t) \in q\inv(L)$ and so $\a(t)= q\ci \hat{\a}(t)\in L$. Moreover 
\begin{align*}
(\hat{\a}^\ast D_{\r_V}f)(t)= (D_{\r_V}f)_{\hat{\a}(t)}(\hat{\a}'(t)) = (df)_{\hat{\a}(t)}(\hat{\a}'(t))= (f\ci \hat{\a})'(t).
\end{align*}
Hence $\int_\R(\hat{\a}^\ast D_{\r_V}f) = \int_{\R}(f\ci \hat{\a})'(t) =0$  for $f\in C^\8_c(P,V)_{\r_V}$, because $f\ci \hat{\a}$ has compact support in $\R$.
\end{definition}

The following Remark is obvious.
\begin{remark}\label{GruIntervall}
Let $W= \bigcup_{i=1}^n I_i$ be a union of finitely many closed intervals in $\R$. Then $W$ is a submanifold with boundary. In fact let $W = \bigcup_{j\in J}C_j$ be the disjoint union of the connected components of $W$. For $j \in J$ and $x\in C_j$ we find $i_j$ with $x\in I_{i_j}$. Hence $I_{i_j}\subs C_j$. If $j_1 \neq j_2$ then $I_{i_{j_1}}\cap I_{i_{j_2}} = \emptyset$ and so $i_{j_1} \neq i_{j_2}$. Therefore $\#J \leq n$. Obviously the sets $C_j$ are intervals.
\end{remark}

The following lemma is a modification of the proofs of \cite[Lemma 3.7 and Corollary 3.10]{Schuett:2013}.
\begin{lemma}\label{GruInterCoverSub}
Let $(U_i)_{i\in\N}$ be a relative compact open cover of $\R$ with $U_i \neq \emptyset$. Then there exists an open cover $(W_i)_{i\in \N}$ of $\R$, such that $W_i \subs U_i$, $W_i\neq \emptyset$ and $\ol{W}_i$ is a submanifold with boundary.
\end{lemma}
\begin{proof}
Let $K_n:=[-n,n]$ for $n \in \N$. For all $x\in K_1$ there exists $i_x\in \N$, such that $x\in U_{i_x}$. Let $\ol{B}_{\ep_x}(x)\subs U_{i_x}$. We find  $x_1,\dots,x_{N_1}$, such that $K_1 \subs \bigcup_{k=1}^{N_1}B_{\ep_{x_k}}(x_k)$. We define $V_{1,k}:= B_{\ep_{x_k}}(x_k)$ and $U_{i_{1,k}}:=U_{i_{x_k}}$ for $k=1,\dots, N_1$. We have $K_1 \subs \bigcup_{k=1}^{N_1}V_{1,k}$ and $\ol{V}_{1,k}\subs U_{i_{1,k}}$. We can argue analogously for the compact set $K_n\setminus K_{n-1}$  with $n\geq 2$ and find open intervals $V_{n,1},\dots,V_{n,N_n}$ and indices $i_{n,k}$ such that $K_n\setminus K_{n-1} \subs \bigcup_{k=1}^{N_n}V_{n,k}$ and $\ol{V}_{n,k}\subs U_{i_{n,k}}$. We obtain $\R \subs \bigcup_{n=1}^\8 \bigcup_{k=1}^{N_n} V_{n,k}$. For $i \in \N$ we define $I_i:=\set{(n,k): i_{n,k}=i}$. We have $\#I_i<\8$, because $U_i$ is relative compact. Now we define
\begin{align*}
W_i:= 
\begin{cases}
\bigcup_{(n,k) \in I_i}V_{(n,k)} &:I_i\neq \emptyset\\
J,
\end{cases}
\end{align*}
where $J$ is an arbitrary not degenerated interval that is contained in $U_i$.  We obtain $\bigcup_{i\in \N}W_i= \R$ and $W_i \subs U_i$ for all $i\in \N$. Moreover $W_i$ is a finite union of open intervals. Let $W_i= \bigcup_{j=1}^nJ_j$ with intervals $J_j$. We have $\ol{W}_i= \bigcup_{j=1}^n\ol{J}_j$. Hence $\ol{W}_i$ is a manifold with boundary (see Remark \ref{GruIntervall}). 
\end{proof}

In the following lemma we use the concept of weak products of infinite-dimensional Lie groups (cf. \cite[Section 7]{Glockner:2003} respectively \cite[Section 4]{Glockner:2007}). 
\begin{lemma}\label{GruDerPullBackSmooth}
In the situation of Definition \ref{GruDerPullBack} the group homomorphism $$\hat{\a}_G^\ast \co C^\8_c(P,G)_{\r_G}  \ra C^\8_c(\R, G),~ \ph\ms \ph\ci \hat{\a}$$
is in fact a Lie group homomorphism such that the corresponding  Lie algebra homomorphism is given by $\hat{\a}^\ast_\fg \co C^\8_c(P,\fg)_{\r_\fg} \ra C^\8_c(\R, \fg),~ f\ms f\ci \hat{\a}$.
\end{lemma}
\begin{proof}
Using the construction of the Lie group structure described in \cite[Chapter 4]{Schuett:2013}, we can argue in the following way.  
Let $(\ol{V}_i,\s_i)_{i\in \N}$ be a locally finite compact trivialising system of $H\hra P \xra{q}M$ (see \cite[Definition 3.6 and Corollary 3.10]{Schuett:2013}). We define $U_i:= \a\inv(V_i)$ for $i \in \N$. The map $\a$ is proper. Because $(\a\inv(\ol{V}_i))_{i\in \N}$ is a compact locally finite cover of $\R$, also $(\ol{U}_i)_{i\in \N}$ is a compact locally finite cover of $\R$. We use  Lemma \ref{GruInterCoverSub} and find a cover $(W_i)_{i\in \N}$ of $\R$ such that $W_i \subs U_i$ and $(\ol{W}_i )_{i\in \N}$ is a compact locally finite cover  of $\R$ with submanifolds with boundaries. Moreover we have $\ol{W}_i \subs \a\inv(\ol{V}_i)$ for all $i \in \N$. Now $(\ol{W}_i, \id|^\R_{\ol{W}_i})_{i\in \N}$ is a compact locally finite trivialising system of the trivial principal bundle $\set{1}\hra \R \xra{\id}\R$ with the trivial action $\set{1}\ti G \ra G$. We get the following diagram
\begin{align}\label{GruKomDiaGlatt}
\begin{xy}
\xymatrixcolsep{5pc}\xymatrix{
C^\8_c(P,G)_{\r_G} \ar@{_{(}->}[d]_{f\ms (f\ci \s_i)_{i}} \ar[r]^{\hat{\a}^\ast_G}& C^\8_c(\R,G) \ar@{_{(}->}[d]^{f\ms (f|_{W_i})_{i}}\\
\prod^\ast_{i\in \N} C^\8(\ol{V}_i,G) \ar[r]^{(\ps_i)_{i\in \N}}& \prod^{\ast}_{i\in \N} C^\8(\ol{W}_i,G)
}
\end{xy}
\end{align}
where the group homomorphisms $\ps_i$ are given by the diagram
\begin{align*}
\begin{xy}
\xymatrixcolsep{5pc}\xymatrix{
C^\8(\ol{V}_i,G) \ar[r]^{\ps_i} \ar[d]_{\t_i}& C^\8(\ol{W}_i,G)\\
C^\8(\ol{V}_i\ti H, G) \ar[d]_{f\ms f\ci \ph_i}&\\
C^\8(P|_{\ol{V}_i},G)\ar[uur]_{f\ms f\ci \hat{\a}|_{\ol{W}_i}}&
}
\end{xy}
\end{align*}
with $\t_i \co C^\8(\ol{V}_i,G)\ra C^\8(\ol{V}_i\ti H, G),~f\ms ((x,h)\ms \r_G(h).f(x))$ and $\ph_i$ the inverse of $\ol{V}_i\ti H \ra P|_{\ol{V}_i},~ (x,h)\ms \s_i(x)h$. Defining $\ta^i\co \ol{W}_i \ra\ol{V}_i\ti H,~ \ta^i:=  \ph_i\ci \hat{\a}|_{\ol{W}_i}$ and $\ta^i_j:=\pr_j\ci \tau^i$ for $j \in \set{1,2}$ the map $\ps_i \co C^\8(\ol{V}_i,G) \ra C^\8(\ol{W}_i, G)$ is given by 
\begin{align*}
f \ms {\r}_G (\pr_2\ci \ph_i \ci \hat{\a}|_{\ol{W}_i}(\bl)). (f\ci \pr_1\ci \ph_i \ci \hat{\a}|_{\ol{W}_i}(\bl)) = {\r}_G(\ta^i_2(\bl)).(f\ci \ta^i_1(\bl)).
\end{align*}
In order to show that (\ref{GruKomDiaGlatt}) is commutative let $f\in C^\8_c(P,G)_{\r_G}$. Then
\begin{align*}
\r_G(h).f\ci\s_i(x) = f(\s_i(x).h) =f(\ph_i\inv(x,h)) 
\end{align*}
for all $(x,h) \in \ol{V}_i \ti H$. Hence $\ps_i(f\ci \s_i) = f\ci\hat{\a}|_{\ol{W}_i}$.
To show that $\ps_i$ is a Lie group homomorphism it is enough to show that $C^\8(\ol{V}_i, G) \ti \ol{W}_i \ra G,~ (f,x)\ms \r_G(\ta_2(x), f(\ta_1(x)))$ is smooth (\cite{Alzaareer:2013} respectively \cite[Theorem 2.25]{Schuett:2013}). The map $C^\8(\ol{V}_i,G) \ti \ol{V}_i,~ (f,y) \ms f(y)$ is smooth (see \cite{Alzaareer:2013} respectively \cite[Theorem 2.26]{Schuett:2013}) and so $C^\8(\ol{V}_i,G) \ti \ol{W}_i \ra H\ti G,~(f,x)\ms (\ta_2(x), f(\ta_1(x)))$ is smooth.
It is left to show that $L(\hat{\a}^\ast_G)$ is given by $C^\8_c(P,\fg)_{\r_\fg} \ra C^\8_c(\R,\fg),~ f\ms f\ci \hat{\a}$. To this end let $f\in C^\8_c(P,\fg)_{\r_\fg}$. We calculate 
\begin{align*}
&L(\hat{\a}^\ast_G)(f) = \fr{\partial}{\partial t}\bigg|_{t=0} \hat{\a}_G^\ast(\exp(tf)) = \fr{\partial}{\partial t}\bigg|_{t=0} (\exp(tf) \ci \hat{\a})\\
= &\fr{\partial}{\partial t}\bigg|_{t=0} (\exp_G\ci (t\.f) \ci \hat{\a}) \ub{=}{\ast} \fr{\partial}{\partial t}\bigg|_{t=0} ( (t\.f) \ci \hat{\a}) = f\ci \hat{\a}
\end{align*}
where $\ast$ follows from 
\begin{align*}
\ev_p\left(\fr{\partial}{\partial t}\bigg|_{t=0} (\exp_G\ci (t\.f) \ci \hat{\a}) \right)= \fr{\partial}{\partial t}\bigg|_{t=0} \left(\exp_G (t\.f (\hat{\a}(p))\right) = \ev_p\left(\fr{\partial}{\partial t}\bigg|_{t=0} ( (t\.f) \ci \hat{\a}) \right)
\end{align*}
for $p \in P$. Now the assertion follows from \cite[Proposition 4.5]{Glockner:2007} respectively \cite[Corollary 2.38]{Schuett:2013}.
\end{proof}

\begin{definition}[Cf. Proof of Lemma V.10 in \cite{Neeb:2004}]
We define the cocycle 
\begin{align*}
&\o_\R \co C^\8_c(\R,\fg)^2 \ra  \ol{\Omega}_c^1(\R, V)= H_{dR,c}^1(\R,V) \ra V,\\
& (f,g) \ms [\k_\fg(f,g')] \ms \int_\R \k_\fg(f(t), g'(t)) dt.
\end{align*} 
\end{definition}

The following Lemmas \ref{GruLemma1}, \ref{GruLemmaInteggM} and \ref{GruLemma2} are used to proof Lemma \ref{GruWichtigLemma}, a generalisation of \cite[Lemma V.16]{Neeb:2004} from the case of a current group to the  case of a group of sections.

In the case of a compact base manifold, there exists a corresponding statement to the following Lemma \ref{GruLemma1}. It is given by  equation (9) in \cite[Remark 4.3]{Neeb:2009}. 
\begin{lemma}\label{GruLemma1}
Given $x_0 \in M$, $p_0 \in P_{x_0}$ and $\a \in C^\8_p(\R, M)$ with $\a(0)=x_0$ we get
\begin{align}\label{GruGL1CCC}
I_\a \ci \o_M = \o_\R \ci (\hat{\a}^\ast_\fg \ti \hat{\a}^\ast_\fg ).
\end{align}
Hence the following diagram commutes:
\begin{align*}
\begin{xy}
\xymatrixcolsep{5pc}\xymatrix{
C^\8_c(P, \fg)_{\r_\fg}^2 \ar[r]^{\o_M} \ar[d]_{\hat{\a}^\ast_\fg \ti \hat{\a}_\fg^\ast}& \ol{\Omega}_c^1(P,V)_{\r_V}^\hor  \ar[d]^{I_\a} \\
C^\8_c(\R , \fg)^2 \ar[r]^{\o_\R} & V.
}
\end{xy}
\end{align*}
\end{lemma}
\begin{proof}
For $g\in C^\8_c(P, \fg)_{\r_\fg}$ we have
\begin{align*}
(\hat{\a}^\ast D_{\r_\fg}g)(t) = D_{\r_\fg} g(\hat{\a}'(t)) = (g\ci \hat{\a})'(t),
\end{align*}
because $\hat{\a}$ is a horizontal map. For $f,g\in C^\8_c(P, \fg)_{\r_\fg}$ we get
\begin{align*}
&I_\a(\o_M (f,g)) =  I_\a([\k_\fg(f,D_{\r_\fg}g)]) = \int_\R\hat{\a}^\ast \k_\fg(f,D_{\r_\fg}g) = \int_\R \k_\fg(f\ci \hat{\a} , \hat{\a}^\ast D_{\fg}g)\\
=& \int_\R \k_\fg(f\ci \hat{\a}(t) , (g\ci \hat{\a})'(t)) dt = \o_\R \ci (\hat{\a}^\ast_\fg \ti \hat{\a}^\ast_\fg)(f,g).
\end{align*}
\end{proof}

The following Lemma \ref{GruLemmaInteggM} can be found in \cite[Remark C.2 (a)]{Neeb:2009}.
\begin{lemma}\label{GruLemmaInteggM}
Let $\ph \co G_1 \ra G_2$ be a Lie group homomorphism and $\fg_i$ the Lie algebra of $G_i$ for $i \in \set{1,2}$. Moreover let $V$ be a trivial $\fg_i$-module and $\o \in Z^2_c(\fg_2,V)$. Then we get
\begin{align}\label{GruGL2IntegMapCCC}
\per_{\o} \ci \p_2(\ph) = \per_{L(\ph)^\ast \o}
\end{align}
as an equation in the set of group homomorphism from $\p_2(G_1)$ to $V$. 
\end{lemma}

The following lemma corresponds to the first equation in \cite[Remark 4.3 (10)]{Neeb:2009}.
\begin{lemma}\label{GruLemma2}
Let $x_0\in M$ and $p_0 \in P_{x_0}$ be base points and $\a \in C_p^\8(\R,M)$. Then
\begin{align}\label{GruGL4IntegMapCCC}
I_\a \ci \per_{\o_M} =\per_{I_\a \ci \o_M} \co \p_2(C^\8_c(P,G)_{\r_G}) \ra V.
\end{align}
\end{lemma}
\begin{proof}
Let $\o_M^l \in \Omega^2(C^\8_c(P,G)_{\r_G} , \ol{\Omega}^1_c(P,V)_{\r_V}^\hor)$ be the corresponding left invariant  2-form of $\o_M \in Z^2_{ct}(C^\8_c(P,\fg)_{\r_{\fg}}, \ol{\Omega}^1_c(P,V)_{\r_V}^\hor)$. Then $I_\a \ci \o_M^l \in \Omega^2_{dR}(C^\8_c(P,K)_{\r_K},V)$ is left invariant and 
\begin{align*}
(I_\a \ci \o_M^l)_1(f,g) = I_\a ((\o_M^l)_1(f,g)) =I_\a \ci \o_M(f,g)
\end{align*}
for $f,g\in C^\8_c(P,\fg)_{\r_{\fg}}=T_1C_c^\8(P,G)_{\rho_G}$. Hence $(I_a \ci \o_M)^l = I_\a \ci \o_M^l$.
For $[\s]\in \p_2(C^\8_c(P,G)_{\r_G})$ with a smooth representative $\s$ we get 
\begin{align*}
\per_{I_\a \ci \o_M}([\s]) = \int_{\mathbb{S}^2}\s^\ast (I_\a \ci \o_M^l) = \int_{\mathbb{S}^2} I_\a \ci \s^\ast \o_M^l = I_\a \ci \int_{\mathbb{S}^2} \s^\ast \o_M^l = I_\a\ci \per_{\o_M} ([\s]).
\end{align*}
\end{proof}

The following lemma is a generalisation of \cite[Lemma V.16]{Neeb:2004} in the case of a finite-dimensional Lie group.
\begin{lemma}\label{GruWichtigLemma}
For a proper map $\a \in C_p^\8(\R,M)$ and the base points $x_0\in M$ and $p_0 \in P_{x_0}$ we get
the following diagram commutes:
\begin{align*}
\begin{xy}
\xymatrixcolsep{5pc}\xymatrix{
\p_2(C^\8_c(P,G)_{\r_G}) \ar[r]^-{\per_{\o_M}} \ar[d]_-{\p_2(\hat{\a}^\ast_G)} & \ol{\Omega}_c^1(P,V)_{\r_V}^{\text{hor}} \ar[d]^-{I_\a}\\
\p_2(C^\8_c(\R,G)) \ar[r]_-{\per_{\o_{\R}}} & V.
}
\end{xy}
\end{align*}
\end{lemma}
\begin{proof}
We calculate
\begin{align}\label{GruGL3IntegMapCCC}
I_\a \ci \per_{\o_M} \ub{=}{(\ref{GruGL4IntegMapCCC})} \per_{I_\a \ci \o_M} \ub{=}{(\ref{GruGL1CCC})} \per_{\o_{\R}\ci (\hat{\a}^\ast \ti \hat{\a}^\ast)} \ub{=}{(\ref{GruGL2IntegMapCCC})} \per_{\o_{\R}} \ci \p_2(\hat{\a}_G^\ast).
\end{align}

\end{proof}

\begin{lemma}\label{GruIntegralLemma}
Let $x_0 \in M$ and $p_0 \in P_{x_0}$ be base points and $\a \in C_p^\8(\R,M)$ with $\a(0)=x_0$. Moreover, let $\hat{\a} \in C^{\8}(\R,P)$ be the unique horizontal lift of $\a$ to $P$ with $\hat{\a}(0)=p_0$. 
\begin{compactenum}
\item We have the commutative diagram
\begin{align*}
\begin{xy}
\xymatrixcolsep{5pc}\xymatrix{
H_{dR,c}^1(M,V_\fix) \ar[d]_-{\cong}^{q^\ast} \ar[r]^-{I_\a}_-{[\t] \ms \int_\R \a^\ast \t} & V \ar[d]^-{\id}\\
H_{dR,c}^1(P , V)_{\r_V} \ar[r]^-{I_\a}_-{[\t] \ms \int_\R \hat{\a}^\ast  \t} & V.
}
\end{xy}
\end{align*}
\item Given $[\t] \in H^1_{dR,c}(M,V_\fix)$ we have  $\int_\R \hat{a}^\ast q^\ast \t = \int_{\R} \a^\ast \t$ respectively
\begin{align*}
\int_\R \hat{\a}^\ast \t = \int_\R \a^\ast q_\ast \t
\end{align*}
for all $[\t] \in H_{dR,c}^1(P,V)_{\r_V}$.
\end{compactenum}
\end{lemma}
\begin{proof}
\begin{compactenum}
\item 
Given $[\t] \in H^1_{dR,c}(M,V_\fix)$, we calculate
\begin{align*}
\int_\R \hat{\a}^\ast (q^\ast \t) =  \int_\R (q \ci \hat{\a})^\ast  \t) = \int_\R \a^\ast  \t.
\end{align*}
\item Clear.
\end{compactenum}
\end{proof}

The following Lemma \ref{GruDiskLemma} comes from \cite[Corrolary IV.21]{Neeb:2004}.
\begin{lemma}\label{GruDiskLemma}
If $\gg \subs V$ a discrete subgroup then
\begin{align*}
H^1_{dR,c}(M,\gg) := \set{[\t] \in H^1_{dR,c}(M,V): (\forall \a \in C^\8_p(\R,M)) \int_\R \a^\ast \t \in \gg }
\end{align*}
is a discrete subgroup of $\ol{\Omega}^1_c(M,V)$.
\end{lemma}

The following statement can be found in the proof of \cite[Proposition V.19]{Neeb:2004}.
\begin{lemma}\label{GruRLemma}
Because $\k_\fg$ is universal and $V$ is finite dimensional it is well known that $\gp_{\o_\R}=\im(\per_{\omega_\R})$ is a discrete subgroup of $\ol{\Omega}^1_c(\R,V) = H_{dR,c}^1(\R,V) \cong V$. 
\end{lemma}
\begin{proof}
We argue exactly like in the proof of \cite[Proposition V.19]{Neeb:2004}, by combining \cite[Theorem II.9]{Maier:2003} and \cite[Lemma V.11]{Neeb:2004}.
\end{proof}

\begin{remark}\label{GruproppeperinPFP}
Because $\ol{q} \co \ol{P} \ra M$ is a finite covering, $\ol{q}$ is a proper map and so a curve $\ol{\a}\co  \R \ra \ol{P}$ is proper  if and only if $\ol{q} \ci \ol{\a} \co \R \ra M$ is proper. Hence the maps in $C^\8_p(\R,\ol{P})$ are proper in the usual sense.
\end{remark}

\begin{lemma}\label{GruIntegralLemmaPiStern}
Let $\ol{\a}\co \R \ra \ol{P}$ be a proper map. We define $\ol{x}_0:= \ol{\a}(0)$, $x_0:=\ol{q}(\ol{x}_0)$ and $\a := \ol{q} \ci \ol{\a}$. Moreover let $p_0 \in P$ with ${\p}(p_0)=\ol{x}_0$ and  $\hat{\a}\co \R \ra P$ be the unique horizontal lift of $\a$ to $P$ with $\hat{\a}(0)= p_0$ then
\begin{align*}
\begin{xy}
\xymatrixcolsep{5pc}\xymatrix{
\ol{\Omega}_c^1(\ol{P},V)_{\ol{\r}_V} \ar[d]_-{\cong}^{\p^\ast} \ar[r]_-{[\t] \ms \int_\R \ol{\a}^\ast \t} & V \ar[d]^-{\id}\\
\ol{\Omega}_c^1(P , V)_{\r_V}^\hor  \ar[r]^-{I_\a}_-{[\t] \ms \int_\R \hat{\a}^\ast  \t} & V
}
\end{xy}
\end{align*}
commutes.
\end{lemma}
\begin{proof}
We have $\p \ci\hat{\a} = \ol{\a}$ because $\ol{\a}$ is the unique horizontal lift of $\a$ to $\ol{P}$ with $\ol{\a}(0)= \ol{x}_0$ and $\p\ci \hat{\a}$ is also a horizontal lift of $\a$ to $\ol{P}$ that maps $0$ to $\ol{x}_0$.
Hence
\begin{align*}
\int_\R \hat{\a}^\ast (\p^\ast \t) =  \int_\R (\p \ci \hat{\a})^\ast  \t = \int_\R \ol{\a}^\ast  \t
\end{align*}
for $\t \in \Omega^1_c(\ol{P},V)_{\ol{\r}_V}$.
\end{proof}

The proof of the following lemma is similar to the proof of \cite[Lemma A.1]{Neeb:2004}.
\begin{lemma}\label{GruKompaktRegulaer}
Given a compact set $L \subs C^\8_c(P,G)_{\r_G}$ we find a compact set $K\subs M$ such that $L \subs C^\8_K(P,G)_{\r_G}$.
\end{lemma}
\begin{proof}
From \cite[Theorem 4.18]{Schuett:2013} we know that the map $\exp_\ast \co C^\8_c(P,\fg)_{\r_\fg} \ra C^\8_c(P,G)_{\rho_G},~ f\ms \exp_G \ci f$ is a local diffeomorphism around $0$. Given a compact set $K\subs M$ we have 
\begin{align}\label{GruExpPasst}
\exp_\ast(C^\8_K(P,\fg)_{\r_\fg}) \subs C^\8_K(P,G)_{\r_G}.
\end{align}
Let $U\subs C^\8_c(P,G)_{\r_G}$ be a $1$-neighbourhood and $V\subs C^\8_c(P,\fg)_{\r_\fg}$ be a $0$-neighbourhood such that $\exp_\ast|_V^U$ is a diffeomorphism. We write $\gph:= \left(\exp_\ast|_V^U\right)\inv$. If $L \subs U$ is a compact set, then $\gph (L)$ is a compact subset $C^\8_c(P,\fg)_{\r_\fg}$. Because $C^\8_c(P,\fg)_{\r_\fg}$ is a strict LF-space we find a compact subset $K\subs M$ such that $\gph (L)\subs C^\8_K(P,\fg)_{\r_{\fg}} \cap V$ (see \cite[Theorem 6.4]{Wengenroth:2003} or \cite[Remark 6.2 (d)]{Glockner:2008a}). Hence with (\ref{GruExpPasst}) we get $L \subs C^\8_K(P,G)_{\r_G}$. 
Now let $L \subs C^\8_c(P,G)_{\r_G}$ be an arbitrary compact subset. Let $W\subs C^\8_c(P,G)_{\r_G}$ be a $1$-neighbourhood such that $\ol{W}\subs U$. Because $L$ is compact we find $n\in \N$ and $g_i \in C^\8_c(P,G)_{\r_G}$ such that $L \subs \bigcup_{i=1}^n g_i\. \ol{W}$. Defining the compact set $L_i:= L\cap g_i\. \ol{W}$, we get $L\subs \bigcup_{i=1}^n L_i$. Let $i \in \set{1,\:,n}$. We get $g_i\inv \cdot L_i \subs \ol{W} \subs U$. Hence we find a compact set $K_1\subs  M$ with   $g_i\inv \cdot L_i \subs C^\8_{K_1}(P,G)_{\r_G}$. Let $K_2\subs M$, be compact with $\supp(g_i)\subs q\inv(K_2)$ and $K_i:= K_1\cup K_2$ and $K:= \bigcup_{i=1}^n K_i$. We get 
\begin{align*}
L_i \subs g_i \cdot C^\8_{K_1}(P,G)_{\r_G} \subs C^\8_{K_i}(P,G)_{\r_G} \subs C^\8_{K}(P,G)_{\r_G}.
\end{align*}
Hence $L= \bigcup_{i=1}^n L_i \subs C^\8_{K}(P,G)_{\r_G}$.
\end{proof}

In \cite[Remark IV.17]{Neeb:2004} (where $M$ is non-compact)  Neeb extends a smooth loop $\a \co [0,1]\ra M$ by a smooth proper map $\g \co [0,\8[ \ra M$ to a proper map $\til{a} \co \R \ra M$ such that for all 1-forms $\t$ with compact support one gets $\int_\a\t= \int_{\til{\a}} \t$. This construction is also used in the proof of the following Theorem \ref{GruAllFormsAreClosed}. An analogous theorem was proofed in the compact case in \cite[Proposition 4.11]{Neeb:2009}.
\begin{theorem}\label{GruAllFormsAreClosed}
If $M$ is non-compact, we have 
\begin{align*}
\im(\per_{\omega_M}) = \gp_{\o_M} \subs H^1_{dR,c}(M,\mathbb{V}) \subs \ol{\Omega}^1_c(M,\mathbb{V}).
\end{align*}
This means that all forms in $\gp_{\o_M}$ are closed.
\end{theorem}
\begin{proof}
Because $\p^\ast \co {\Omega}^\bl_c(\ol{P},V)_{\ol{\r}_V} \ra  {\Omega}^\bl_c(P,V)_{\r_V}^\hor$ is  an isomorphism of chain complexes, it is enough to show $\gp_{\o_M} \subs H^1_{dR,c}(\ol{P},V) \subs \ol{\Omega}^1_c(\ol{P},V)$. To this end let $[\t] \in \gp_{\o_M}$  and $\ol{\a}_0, \ol{\a}_1 \co [0,1]\ra \ol{P}$ be closed smooth curves in a point $\ol{x}_0 \in \ol{P}$ that are homotopic  relative $\set{0,1}$ by a smooth homotopy $\ol{F}\co [0,1]^2 \ra \ol{P}$. From Lemma \ref{GrugeschlForm} we get that it is enough to show $\int_{\ol{\a}_0} \t=\int_{\ol{\a}_1} \t$.
By composing $\ol{\a}_i$ respectively $\ol{F}(s,\bl)$ with a strictly increasing diffeomorphism $\ph \co [0,1]\ra [0,1]$ whose jets  vanish in  $0$ and $1$, respectively, we can assume that in a local chart all derivatives of $\ol{\a}_i$ and $\ol{F}(s,\bl)$ vanish in $0$ and $1$ respectively, because $\int_{\ol{\a}_i}\t = \int_{\ol{\a}_i\ci \ph}\t$ (forward parametrization dose not change line integrals).
Because $M$ is non-compact, we find a proper map $\ol{\g} \co [0,\8[ \ra \ol{P}$ such that $\g(0)=\ol{x}_0$ and in a local chart all derivatives vanish in $0$ (see \cite[Lemma IV. 5]{Neeb:2004} and composition with $x\ms x^3$). For $i \in \set{0,1}$ we define the smooth map
\begin{align*}
\ol{\a}_i^\R \co \R \ra \ol{P},~ t \ms 
\begin{cases}
\ol{\g}(-t) &:t<0\\
\ol{\a}_i(t)&:t \in [0,1]\\
\ol{\g}(t-1)&: t>1.
\end{cases}
\end{align*}
Moreover, we define the smooth homotopy
\begin{align*}
\ol{F}^\R \co [0,1] \ti \R \ra \ol{P}, ~ (s,t) \ms 
\begin{cases}
\ol{\g}(-t) &:t<0\\
\ol{F}(s,t)&:t \in [0,1]\\
\ol{\g}(t-1)&: t>1.
\end{cases}
\end{align*}
Hence we have $\ol{\a}_i^\R, \ol{F}^\R(s,\bl) \in C^\8_p(\R,\ol{P})$ for $i \in \set{0,1}$ and $s\in [0,1]$ (see Remark \ref{GruproppeperinPFP}). We define $\a_i:=\ol{q}\ci\ol{\a}_i$, $F:=\ol{q}\ci \ol{F}$, $\a_i^\R:= \ol{q}\ci \ol{\a}^\R_i$, $F^\R:= \ol{q}\ci \ol{F}^\R$, $\g:= \ol{q}\ci \ol{\g}$ and $x_0:=\ol{x}_0$.  
The curves $\a_0$ and $\a_1$ are closed curves in $x_0$ and are homotopic relative $\set{0,1}$ by the homotopy $F$, because $\a_i(j)=\ol{q}(\ol{x}_0) =x_0$ and $F(i,\bl)=\ol{q}\ci \ol{F}(i,\bl)=\ol{q}\ci \ol{\a}_i = \a_i$ for $j,i \in \set{0,1}$. Moreover
\begin{align*}
\a_i^\R(t)=
\begin{cases}
\g(-t) &:t<0\\
\a_i(t)&:t \in [0,1]\\
\g(t-1)&: t>1,
\end{cases}
\\
F^\R(s,t)=
\begin{cases}
\g(-t) &:t<0\\
F(s,t)&:t \in [0,1]\\
\g(t-1)&: t>1
\end{cases}
\end{align*}
and $\a_i^\R, F^\R(s,\bl) \in C^\8_p(\R,M)$.
We choose $p_0 \in \p\inv(\set{\ol{x}_0})$. Now let $\hat{\a}_i^\R \co \R \ra P$ be the unique horizontal lift of $\a_i^\R$ to $P$ with $\hat{\a}^\R_i(0)=p_0$ and $\hat{F}^\R\co [0,1]\ti \R \ra P$ be the unique horizontal lift of $F^\R$ to $P$ such that $\hat{F}^\R(s,0) = p_0$ for all $s\in [0,1]$. The map $\hat{F}^\R$ is not a homotopy relative $\set{0,1}$, but we have $\hat{F}^\R(0,s) = \hat{\a}_0^\R(s)$ and $\hat{F}^\R(1,s) = \hat{\a}_1^\R(s)$ for all $s \in [0,1]$. For $i\in \set{0,1}$ we have
\begin{align}\label{GruIntegGleichung}
\int_{\ol{\a}_i}\t =\int_{\ol{\a}_i^\R}\t = \int_{\hat{\a}^\R_i} \p^\ast \t,
\end{align}
where the last equation follows from Lemma \ref{GruIntegralLemmaPiStern}.  Because of (\ref{GruIntegGleichung}) and Lemma \ref{GruWichtigLemma} it is enough to show $\p_2((\hat{\a}_0^\R)^\ast)= \p_2((\hat{\a}_1^\R)^\ast)$ as group homomorphisms from $\p_2(C^\8_c(P,G)_{\r_G})$ to $\p_2(C^\8_c(\R,G))$. From \cite[Theorem A.7]{Neeb:2004} we get $\p_2(C^\8_c(\R,G)) = \p_2(C_c(\R,G))$. We set $I:=[0,1]$. Let $\s \co I^2 \ra C^\8_c(P,G)_{\r_G}$ be continuous with $\s|_{\partial I^2} =c_{1_G}$. Because $\p_2((\hat{\a}_i^\R)^\ast)([\s]) = [\s(\bl)\ci \hat{\a}_i^\R]$ for $i\in \set{0,1}$ it is enough to show
\begin{align*}
[\s(\bl)\ci \hat{\a}_0^\R] = [\s(\bl)\ci \hat{\a}_1^\R]
\end{align*}
in $\p_2(C_c(\R,G))$. Hence we have to construct a continuous map $H\co [0,1]\ti I^2 \ra C_c(\R,G)$ with $H(0,\bl)= \s(\bl) \ci \hat{a}^\R_0$, $H(1,\bl)=\s(\bl) \ci \hat{\a}^\R_1$ and $H(s,x) =c_{1_G}$ for all $s\in [0,1]$ and $x\in \partial I^2$.
We define $H(s,x)= \s(x)\ci \hat{F}^\R(s,\bl)$ for $s\in [0,1]$ and $x \in I^2$. Because $\sigma|_{\partial I^2}= c_{1_G}$, it is left to show that $H$ is continuous. Let $K\subs M$ be compact such that $\im(\s) = \s(I^2) \subs C^\8_K(P,G)_{\r_G}$ (see Lemma \ref{GruKompaktRegulaer}). For $f\in C^\8_K(P,G)_{\r_G}$ we have $\supp(f\ci \hat{\a}^\R_i) \subs {\a_i^\R}\inv(K)$ as well as $\supp(f\ci \hat{F}^\R(s,\bl)) \subs F^\R(s,\bl)\inv(K)$ for $s\in [0,1]$. Hence $\supp(\s(x)\ci \hat{F}^\R(s,\bl)) \subs F^\R(s,\bl)\inv(K)$ for $x\in I^2$ and $s\in [0,1]$. We have
\begin{align*}
&{F^\R}\inv(K) = F^\R|_{[0,1]\ti [0,1]}\inv (K) \cup F^\R|_{[0,1]\ti]-\8,0]}\inv(K) \cup F^\R|_{[0,1]\ti[1,\8[}\inv(K)\\
=&F^\R|_{[0,1]\ti [0,1]}\inv (K) \cup ([0,1] \ti -\g\inv(K)) \cup ([0,1] \ti \g\inv(K)+1).
\end{align*}
Hence ${F^\R}\inv (K)\subs [0,1]\ti \R$ is compact. Therefore 
\begin{align*}
L:= \bigcup_{s\in [0,1]} F^\R(s,\bl)\inv (K) = \pr_2({F^\R}\inv(K))\subs \R
\end{align*}
is compact. We have $\supp(\s(x)\ci \hat{F}^\R(s,\bl)) \subs L$ for all $x\in I^2$ and $s\in [0,1]$. Thus $\im(H) \subs C_L(\R,G)$. Therefore it is enough to show that $H\co [0,1]\ti I^2 \ra C_L(\R,G) \subs C(\R,G),~ (s,x) \ms \s(x)\ci \hat{F}^\R(s,\bl)$ 
is continuous. We know that $\ta \co [0,1]\ra C(\R,P),~ s\ms \hat{F}^\R(s,\bl)$ is continuous and so the assertion follows from the following commutative diagram
\begin{align*}
\begin{xy}
\xymatrixcolsep{5pc}\xymatrix{
[0,1]\ti I^2 \ar[r]^{\ta \ti \s} \ar[ddr]^H& C(\R,P) \ti C^\8_K(P,G)_{\r_G} \ar@{_{(}->}[d]\\
 & C(\R,P)\ti C(P,G)\ar[d]^{(\a,f)\ms f\ci \a}\\
 &C(\R,G).
}
\end{xy}
\end{align*}
\end{proof}

The following Theorem \ref{GruMainI} corresponds to the Reduction Theorem \cite[Theorem 4.14]{Neeb:2009} (compact base manifold $M$) in the case of a finite-dimensional Lie group $G$ and a finite-dimensional principal bundle $P$. See also Theorem \ref{GruMainI.2}.
\begin{theorem}\label{GruMainI}
The period group $\gp_{\o_M} = \im \per_{\o_M}$ is discrete in $\ol{\Omega}_c^1(M,\mathbb{V})$.
\end{theorem}
\begin{proof}
Because $q^\ast \co H^1_{dR,c}(M,V_\fix) \ra H^1_{dR,c}(P,V)_{\r_V}$ is an isomorphism of topological vector spaces and $\gp_{\o_M} \subs H^1_{dR,c}(M,\mathbb{V}) = H^1_{dR,c}(P,V)_{\r_V}$, it is sufficient to show that $\gp_{\o_M}$ is a discrete sub group of $H^1_{dR,c}(M,V)$ (Lemma \ref{GruIntegralLemma}).
With  Lemma \ref{GruDiskLemma} and Lemma \ref{GruRLemma} it is enough to show
\begin{align}\label{GruTada}
\gp_{\o_M} \subs H_{dR,c}^1(M,\gp_\R).
\end{align}
Let $\b \in \gp_{\o_M}$, $\a \in C^\8_p(\R,M)$ and $[\s]\in \p_2(C^\8_c(P,G)_{\r_G})$ with $\b = \per_{\o_M}([\s])$. Using Lemma \ref{GruIntegralLemma} and Lemma \ref{GruWichtigLemma} we get
\begin{align*}
\int_\R \a^\ast q_\ast \b = \int_\R \hat{\a}^\ast \b = I_\a \ci \per_{\o_M} ([\s]) = \per_{\o_\R} \ci \p_2(\hat{\a}^\ast) ([\s]) \in \gp_{\o_\R}.
\end{align*}
Hence we get (\ref{GruTada}).
\end{proof}

\begin{theorem}\label{GruMainI.2}
If the structure group $H$ and the Lie group $G$ are not finite-dimensional but locally exponential Lie groups that are modelled over locally convex spaces and $\o \co \fg^2 \ra V $ is not necessarily the universal continuous invariant bilinear form but just a continuous invariant bilinear form with values in a Fr{\'e}chet space then $\gp_\o$ is discrete if $\gp_{\o_\R}$ is discrete in $V$. We emphasise that the base manifold $M$ still has to be $\sigma$-compact and finite-dimensional and $\ol{H}$ still has to be finite, but we do not need to assume $\p_2(G)$ to be trivial.
\end{theorem}
\begin{proof}
The only point where it was necessary to assume $G$ to be finite-dimensional and  $\o$ to be universal was in Lemma \ref{GruRLemma}.
\end{proof}

\section{Integration of the Lie algebra action and the main result}\label{Gru1234}

In the case of a compact base manifold (\cite[Section 4.2 (Part about general Lie algebra bundles)]{Neeb:2009}) Neeb and Wockel integrated the adjoined action of $\gg(\frak{G})$ on $\widehat{\gg(\frak{G})}:=\ol{\Omega}^1(M,\mathbb{V})\ti_\o \gg(\frak{G})$ given by 
\begin{align*}
\gg(\frak{G}) \ti \widehat{\gg(\frak{G})} \ra \widehat{\gg(\frak{G})},~
(\eta,([\a],\gamma)) \ms [\eta,([\a],\gamma)]_{\o} = (\o(\eta, \gamma), [\eta,\gamma])
\end{align*}
to a Lie group action of $\gg(\mathcal{G})$ on $\widehat{\gg(\frak{G})}$. As a first step in their proof, Neeb and Wockel  integrated the covariant derivative $d_\frak{G}\co \gg(\frak{G}) \ra \Omega^1(M,\frak{G})$ to a smooth map from $\gg(\mathcal{G})$ to  $\Omega^1(M,\frak{G})$. 
As the absolute derivative is the sum of $d\co C^\8(P,\fg)_{\r_\fg} \ra \Omega^1(P,\fg)$ and $C^\8(P,\fg)_{\r_\fg} \ra \Omega^1(P,\frak{g}),~ f\ms \r_\ast(Z)\we f$, where $Z\co TP \ra L(H)=:\frak{h}$ is the connection form, $\rho_\ast = L(\rho_\fg)\co \frak{h}\ra \der(\fg)$ and $(\r_\ast(Z)\we f)_p(v) = \rho_\ast(Z(v)).f(p)$, they integrated these summands separately. The image of the exterior derivative dose not lie in $\Omega^1(M,\fg)_\fg^\hor$, but in the space $\Omega^1(M,\fg)_\fg$ and in some sense the summand $f\ms \r_\ast(Z)\we f$ annihilates the vertical parts of $df$. The exterior derivative $d\co C^\8(P,\fg)_{\rho_\fg} \ra\Omega^1(M,\fg)_\fg$ integrates to the left logarithmic derivative $\d \co C^\8(P,G)_{\rho_G} \ra\Omega^1(M,\fg)_\fg$, $\ph \ms \d(\ph)$ with $\d(\ph)_p(v)= T\l_{\ph(p)\inv}\ci T\ph(v)$ (\cite{Neeb:2009}). The integration of the second summand is more complicated, and Neeb and Wockel assumed the Lie group $G$ to be $1$-connected (in the special case of the gauge group they did not need this assumption (see \cite[Theorem4.21]{Neeb:2009})).
In the second step they used an exponential law to obtain the integrated action.  
Because our base manifold is not compact the adjoined action of $\gg_c(\frak{G})$ on $\widehat{\gg_c(\frak{G})}:=\ol{\Omega}_c^1(M,\mathbb{V})\ti_{\o_M} \gg_c(\frak{G})$ given by
\begin{align*}
\gg_c(\frak{G}) \ti \widehat{\gg_c(\frak{G})} \ra \widehat{\gg_c(\frak{G})},~
(\eta,([\a],\gamma)) \ms (\o_M(\eta, \gamma), [\eta,\gamma]).
\end{align*}
With the canonical identifications (see Remark \ref{GruRmarkallslsgfla}) the adjoint action is given by 
\begin{gather}\label{GruAdacttitit234}
C^\8_c(P,\fg)_{\r_\fg} \ti (\ol{\Omega}_c^1(P,V)_{\r_V}^\hor \ti_{\o_M} C^\8_c(P,\fg)_{\r_\fg}) \ra \ol{\Omega}_c^1(P,V)_{\r_V}^\hor \ti_{\o_M} C^\8_c(P,\fg)_{\r_\fg}\nonumber\\
(g , ([\a],f)) \ms ([\k_\fg(g,D_{\rho_\fg}(f))], \ad(g,f)).
\end{gather}
We have to integrate this action to a Lie group action of $(\gg_c(\mathcal{G}))_0$ on $\widehat{\gg_c(\frak{G})}$.\footnote{In an earlier version of this paper, we presented a more complicated argumentation for this result. Also our proof required the group $G$ to be semisimple. In this version of the paper we do not assume $G$ to be semisimple.}
Like Neeb and Wockel we have to integrate the covariant derivative $d_\frak{G}\co \gg_c(\frak{G}) \ra \Omega_c^1(M,\frak{G})$ to a smooth map from $\gg_c(\mathcal{G})$ to  $\Omega_c^1(M,\frak{G})$. But we will not describe the absolute derivative via the connection form $Z$ as the sum of the exterior derivative $d$ and the map $f\ms \r_\ast(Z)\we f$. Instead we use the principal connection $HP$ and write $D_{\rho_\fg}= \pr_h^\ast\ci d$, where $\pr_h$ is the projection onto the horizontal bundle and $(\pr_h^\ast\ci d)(f)(v)= df(\pr_h(v))$. In Theorem \ref{Grunctandk} we will show that the map $\Delta:=\pr_h^\ast \ci \d \co C^\8_c(P,G) \ra \Omega^1_c(P,\fg)_{\rho_\fg}^\hor$ is smooth and its derivative in $1$ is given by the absolute derivative $D_{\rho_\fg}$. One could show the smoothness of $\d \co C^\8_c(P,G)_{\rho_{G}} \ra \Omega^1_c(P,\fg)_{\fg}$ and $\pr_h^\ast \co \Omega^1_c(P,\fg)_{\fg} \ra \Omega^1_c(P,\fg)_{\fg}^\hor$ separately, but it is more convenient to show the smoothness of $\Delta$ directly, because we work in the non-compact case and $\Omega^1_c(P,\fg)_{\rho_\fg}^\hor$ is an inductive limit (compare Lemma \ref{Grukljgdfljhg}).
\begin{remark}
In \cite[Chapter 4.2 page 408]{Neeb:2009} Neeb and Wockel define  $\chi^Z(f)v:= \chi(Z(v),(f(p))$ for $f\in \gg\mathcal{G} = C^\8(P, G)_{\r_G}$, $v\in T_pP$, $Z$ the connection form of $P$ and a smooth map $\chi \co \fh \ti G \ra \fg$ that is linear in the first argument. If the connection on $P$ is not trivial, then $TP \ra \fg,~ v \ms \chi(Z(v),(f(p)))$ is not in $\Omega^1(M,\frak{G})= \Omega^1(P,\fg)^\hor_{\r_\fg}$, because it is not horizontal except it is constant $0$. Although the image of the map $\d^\nabla(f) = \d(f)+ \chi^Z(f\inv)$ lies in $\Omega^1(P,\fg)^\hor_{\r_\fg}$, because the image of its derivative in $1$ lies in $\Omega^1(P,\fg)^\hor_{\r_\fg}$ and $\d^\nabla$ is a $1$-cocycle with respect to the adjoint action of $C^\8(P, G)_{\r_G}$ on $\Omega^1(P,\fg)_{\r_\fg}$ and $\Omega^1(P,\fg)^\hor_{\r_\fg}$ is invariant under this action. Although this consideration is not important for our argumentation, because we use a different description of the absolute derivative, as mentioned above.
\end{remark}

\begin{lemma}\label{GruLemmareresk09834t}
Let $V$ be a locally convex space and $G$ a locally convex Lie group. Moreover let $\mu \co G\ti V \ra V$ be  a map that is continuous linear in the second argument.  Let $f\co G \ra V$ be a map that is smooth on an $1$-neighbourhood. If $f(hg) = f(g)+ \mu(g\inv,f(h))$ for $g,h \in G$ then $f$ is smooth.
\end{lemma}
\begin{proof}
Let $U\subs G$ be a $1$-neighbourhood such that $f|_U$ is smooth and $g\in G$. Then $Ug$ is a $g$-neighbourhood  and given $z\in Ug$ we define $h:=  z g\inv\in U$. Hence $z = hg$. Now we calculate
\begin{align*}
&f(z)= f(hg) = f(g)+ \mu(g\inv,f(h)) = f(g)+ \mu(g\inv,f(zg\inv)) \\
= &f(g) + \mu(g\inv,f|_U\ci \varrho_{g\inv}(z)).
\end{align*}
Hence
\begin{align*}
f|_{gU}= f(g)+ \mu(g\inv,\bl)\ci f|_U \ci \varrho_{g\inv}|_{gU}.
\end{align*}
\end{proof}

\begin{lemma}\label{GruAdInvarianzOmega}
We considering the map 
\begin{align*}
\mu\co  C^\8_c(P,G)_{\r_G} \ti \Omega^1_c(P,\fg) \ra \Omega^1_c(P,\fg),~ (\ph,\t) \ra \Ad^G_{\ph}.\t
\end{align*}
with $\Ad^G_{\ph}.\t \co TP \ra \fg,~ v\ms \Ad^G_{\ph(\p(v))}.\t(v)$ and the canonical projection $\pi \co TP \ra P$. The sub space  $\Omega^1(P,\fg)_{\r_\fg}^\hor$ is $\mu$-invariant. In this context $\Ad^G$ means the adjoint action of $G$ on $\fg$.
\end{lemma}
\begin{proof}
Given $\t\in \Omega^1_c(P,\fg)_{\r_\fg}$ and $\ph \in C_c^\8(P,G)_{\r_G}$ we show $\mu(\ph, \t)\in \Omega^1_c(P,\fg)_{\r_\fg}$. Let $h\in H$, $p \in P$ and $v\in T_pP$. We calculate
\begin{align*}
&(R_h^\ast \mu (\ph, \t))_p(v) = \Ad^G(\ph(ph), \t_{ph}(TR_h(v))) = \Ad^G(\r_G(h\inv).\ph(p), \r_\fg(h\inv).\t_p(v))\\
=& T\l_{\r_G(h\inv).\ph(p)} \ci T\varrho_{\r_G(h\inv).\ph(p)\inv} \ci T_1\r_G(h\inv)(\t_p(v))\\
=& T_1(\r_G(h\inv)(\ph(p)) \. \r_G(h\inv)(\bl) \. \r_G(h\inv) (\ph(p)\inv) ) (\t_p(v))\\
=& T_1( \r_G(h\inv)\ci I_{\ph(p)} )(\t_p(v)) = \r_\fg(h\inv)\ci \Ad^G_{\ph(p)} (\t_p(v)),
\end{align*}
where $I_{\ph(p)}(g) = \ph(p)g\ph(p)\inv$ is the conjugation on $G$.
Obviously $\mu(\ph, \t)$ is horizontal if $\t$ is so.
\end{proof}

\begin{definition}
We define the map 
\begin{align*}
\Ad^G_\ast \co C^\8_c(P,G)_{\r_G}^\hor \ti \Omega^1_c(P,\fg)_{\r_\fg}^\hor \ra \Omega^1_c(P,\fg)_{\r_\fg}^\hor,~ (\ph,\t) \ra \Ad^G_{\ph}.\t
\end{align*}
with $\Ad^G_{\ph}.\t \co TP \ra \fg,~ v\ms \Ad^G_{\ph\ci \p(v)}.\t(v)$ and $\pi \co TP \ra P$ the canonical projection.
\end{definition}

\begin{lemma}\label{Gruddkdfjhlemmsdg78}
The map $\Ad^G_\ast \co C^\8_c(P,G)_{\r_G} \ti \Omega^1_c(P,\fg)_{\r_\fg}^\hor \ra \Omega^1_c(P,\fg)_{\r_\fg}^\hor$ is continuous linear in the second argument.
\end{lemma}
\begin{proof}
\begin{compactenum}
Let $\ph \in C^\8_c(P,G)_{\r_G}$ and $K\subs M$ be compact. It is enough to show that 
\begin{align*}
\Ad^G_\ast(\ph,\bl) \co \Omega^1_K(P,\fg)_{\rho_\fg} \ra \Omega^1_K(P,\fg)_{\rho_\fg}
\end{align*}
is continuous, because $\Ad^G_\ast(\ph,\bl)$ is linear and $\Ad^G_\ast(\ph,\bl)(\Omega^1_K(P,\fg)_{\rho_\fg})\subs \Omega^1_K(P,\fg)_{\rho_\fg}$. 
The map $f\co TP \ti \fg \ra \fg,~ (v,w)\ms \Ad^G(\ph\ci\p(v), w)$ is smooth. From \cite[Lemma 4.3.2]{Glockner:a} we know that
\begin{align*}
f_\ast \co C^\8(TP,\fg) \ra C^\8(TP,\fg),~ \t \ms f\ci (\id,\t)
\end{align*}
is continuous.
We can embed  $\Omega^1_K(P,\fg)$ into $C^\8(TP,\fg)$. Hence we are done.
\end{compactenum}
\end{proof}

\begin{definition}\label{GruLeftLogarithmic}
Let $\p \co TP \ra P$ be the canonical projection and $\pr_h\co TP \ra HP$ the projection onto the horizontal bundle.
\begin{compactenum}
\item We define 
\begin{align*}
\d \co C^\8_c(P,G)_{\r_G} \ra \Omega^1(P,\fg),~ \ph\ms \d(\ph)
\end{align*}
with $\d\ph(v) = T\l_{\ph(\p(v))\inv} \ci T\ph (v)$ for $v\in TP$ (cf. \cite[38.1]{Kriegl:1997}).  
\item We define 
\begin{align*}
\pr_h^\ast \co \Omega^1(P,\fg) \ra \Omega^1(P,\fg)^\hor, \t \ms \t \ci \pr_h
\end{align*}
\end{compactenum}
\end{definition}

The statement (b) in the following lemma is well-known and can be found in \cite[38.1]{Kriegl:1997}.
\begin{lemma}\label{Grukdjllr65rglkdf}
\begin{compactenum}
\item We have 
\begin{align*}
\d(C^\8_c(P,G)_{\r_G}) \subs \Omega^1_c(P,\fg)_{\rho_\fg}
\end{align*}
\item Given $f,g \in C^\8_c(P,G)_{\r_G}$ we have 
\begin{align*}
\d(f\cdot g) = \d(g) + \Ad^G_\ast(g\inv, \d(f))
\end{align*}
\end{compactenum}
\end{lemma}
\begin{proof}
\begin{compactenum}
\item 
Let $\ph \in C_c^\8(P,G)_{\rho_G}$, $h \in H$, $p \in P$ and $w\in T_pP$.  We calculate
\begin{align*}
&(R_h^\ast \d(\ph))_p(w) = \d(\ph)_{ph}(TR_h(w)) = T\l_{\ph(ph)\inv} (T\ph (TR_h(w))) \\
=& T(\l_{\ph(ph)\inv} \ci \ph \ci R_h)(w) =:\dagger
\end{align*}
For $x \in P$ we have 
\begin{align*}
&\l_{\ph(ph)\inv} \ci \ph \ci R_h(x) = (\rho_G(h\inv). \ph(p))\inv  \cdot (\rho_G(h\inv). \ph(x)) \\
=& \rho_G(h\inv). (\ph(p)\inv \cdot \ph(x)) 
= \rho_G(h\inv) \ci \l_{\ph(p)\inv} \ci \ph (x)
\end{align*}
We conclude 
\begin{align*}
\dagger = \rho_\fg(h\inv)\ci T\l_{\ph(p)\inv} \ci T\ph(w) = \rho_\fg(h\inv)\ci \d(\ph)_p(w).
\end{align*}
\item The assertion follows directly from \cite[38.1]{Kriegl:1997}.
\end{compactenum}
\end{proof}

\begin{definition}
Let $\pr_h\co TP=VP\oplus HP \ra HP$, be the projection onto the horizontal bundle $HP$. We define 
\begin{align*}
\Delta\co  C_c^\8(P,G)_{\rho_G} \ra \Omega^1(P,\fg)_{\rho_\fg}^\hor, \ph \ms \pr^\ast_h \ci \d(\ph) = \d(\ph)\ci \pr_h.
\end{align*}
\end{definition}

As in \cite{Schuett:2013} we use the concept of weak products of infinite-dimensional Lie groups described in \cite[Section 7]{Glockner:2003} respectively \cite[Section 4]{Glockner:2007} in the following considerations. The following lemma is basically \cite[Corrolary 2.38]{Schuett:2013}, but with modified assumptions.
\begin{lemma}\label{Grudkdkdlsjlf098}
For $i \in \N$ let $G_i$ be a locally convex Lie group, $E_i$ a locally convex space and $f_i \co G_i \ra E_i$ be a smooth map such that $f_i(1)=0$. In this situation there exists an open $1$-neighbourhood $U\subs \prod^\ast_{i\in \N} G_i$ such that the map $f\co \prod^\ast_{i\in \N} G_i \ra \bigoplus_{i\in \N} E_i$, $(g_i)_i \ms (f_i(g_i))_{i}$ is smooth on $U$. 
\end{lemma}
\begin{proof}
Given $i \in \N$ let $\ph_i \co U_i \subs G_i \ra V_i \subs \fg_i$ be a chart around $1$ with $\ph_i(1)=0$. We have the commutative diagram
\begin{align*}
\begin{xy}\xymatrixcolsep{9pc}
\xymatrix{
\prod^\ast_{i\in \N}U_i \ar[r]^-{f|_{\prod^\ast_{i\in \N}U_i} = (f_i|_{U_i})_{i\in \N}} \ar[d]_-{(\ph_i)_{i\in \N}} & \bigoplus_{i\in \N} E_i \\
\bigoplus_{i\in \N} V_i \ar[ur]_-{(f_i\ci \ph_i\inv)_{i\in\N}}.
}
\end{xy}
\end{align*}
Now the assertion follows form \cite[Proposition 7.1]{Glockner:2003}.
\end{proof}

\begin{remark}\label{Gruruckdsld}
Given a compact locally finite  trivializing system $(\ol{V_i},\sigma_i)_{i\in \N}$  of  the principal bundle $H\hra P \xra{q}M$ in the sense of \cite[Definition 3.6]{Schuett:2013} respectively \cite{Wockel:2007}. We follow \cite[Remark 3.5]{Schuett:2013} respectively \cite{Wockel:2007} and define as usually the smooth map $\beta_{\sigma_i}\co q\inv(\ol{V}_i) \ra H$  by the equation $\sigma_i(q(p)) \cdot \beta_{\sigma_i}(p)= p$ for all $p \in q\inv(\ol{V}_i)$. Obviously we have $\beta_{\s_i}(ph)= \b_{\s_i}(p)\cdot h$ for all $h\in H$. Moreover we define the smooth cocycle $\beta_{i,j}\co \ol{V}_i\cap \ol{V}_j \ra H$ by the equation $\sigma_i(x)\cdot \b_{i,j}(x) = \s_j(x)$. We have $\b_{i,j}(x)\inv = \b_{j,i}(x)$ and $\b_{\s_i}(p)\inv\cdot \b_{i,j}(q(p)) =  \b_{\s_j}(p)\inv$ for $p \in q\inv(\ol{V}_i\cap \ol{V}_j)$.
\end{remark}

The proof of the following lemma is similar to the proof of \cite[Proposition 4.6]{Schuett:2013}, where beside other results Sch{\"u}tt, constructed a topological embedding from the compactly supported gauge algebra $\mathrm{gau}_c(P,\fg)_\fg$ into a direct sum $\bigoplus_{i\in\N}C^\8(\ol{V}_i,\fg)$ of locally convex spaces. However the following lemma differers from \cite[Proposition 4.6]{Schuett:2013}, because we deal with horizontal differential forms and these need some extra considerations.
\begin{lemma}\label{Grukljgdfljhg}
Let $(\ol{V_i},\sigma_i)_{i\in \N}$ be a compact locally finite trivializing system in the sense of \cite[Definition 3.6]{Schuett:2013}. The map 
\begin{align*}
\Omega_c^1(P,\fg)_{\rho_\fg}^\hor \ra \bigoplus_{i\in \N} \Omega^1(\ol{V}_i,\fg),~ \t \ms (\sigma_i^\ast \t)_{i\in \N} 
\end{align*}
is a topological embedding.
\end{lemma}
\begin{proof}
We define 
\begin{align*}
\Omega_\oplus:=\set{(\eta_i)_i \in \bigoplus_{i\in\N} \Omega^1(\ol{V}_i,\fg): (\eta_i)_x = \rho_\fg(\b_{i,j}(x))\ci  (\eta_j)_x ~\mathrm{for
}~ x\in \ol{V}_i \cap \ol{V}_j } 
\end{align*}
and the map 
\begin{align*}
\Phi\co \Omega_c^1(P,\fg)_{\rho_\fg}^\hor \ra \bigoplus_{i\in \N} \Omega^1(\ol{V}_i,\fg),~ \t \ms (\sigma_i^\ast \t)_{i\in \N}.
\end{align*}
Lets show $\im(\Phi)\subs \Omega_\oplus$. For $x \in \ol{V}_i\cap \ol{V}_j, v \in T_xM$ we have $\s_j(x)=\s_i(x)\b_{i,j}(x)$ and 
\begin{align}\label{Grunbxml121}
Tq(T\s_j(v) - T(R_{\b_{i,j}(x)} \ci\s_i)(v)) = T(q\ci \s_j)(v) - T(q\ci R_{\b_{i,j}(x)} \ci \s_i)(v)  =0. 
\end{align}
Because $\t$ is $\rho_\fg$ invariant and horizontal we can calculate 
\begin{align*}
&(\s_i^\ast \t)_x(v) = \t_{\s_i(x)}(T\s_i(v)) = \rho_\fg(\b_{i,j}(x))\ci \t_{\s_i(x)\b_{i,j}(x)} (TR_{\b_{i,j}(x)}(T\s_i(v)))\\
\ub{=}{(\ref{Grunbxml121})}& \rho_\fg(\b_{i,j}(x))\ci \t_{\s_j(x)} (T\s_j(v)) = \rho_\fg(\b_{i,j}(x))\ci \s_j^\ast \t_{x} (v).
\end{align*}
Analogous to \cite[Proposition 4.6]{Schuett:2013} we can argue as followed: The map $\Phi$ is linear, $\Omega^1_c(P,\fg)_{\rho_{\fg}}^\hor = \varinjlim \Omega_K^1(P,\fg)_{\rho_\fg}^\hor$ and $(\ol{V}_i)_i$ is locally finite, hence the map $\Phi$ is continuous.
Now let $(\l_i)_{i\in\N}$ be a partition of the unity of $M$ that is subordinated to the cover $(V_i)_i$. Given $\e\in\Omega^1(\ol{V}_i,\fg)$ we define $\wtil{\l_i\e} \in \Omega^1(P,\fg)$ by
\begin{align*}
\wtil{\l_i\e}_p(w):=
\begin{cases}
\l_i(q(p))\cdot \rho_\fg(\b_{\s_i}(p)\inv). \e_{q(p)}(Tq(w)) &: p \in q\inv(V_i)\\
0&: ~\mathrm{else}.
\end{cases}
\end{align*}
With Remark \ref{Gruruckdsld} we get $\wtil{\l_i\e}\in \Omega^1_{\supp(\l_i)}(P,\fg)_{\rho_\fg}^\hor$ and $\sum_{i\in\N} \wtil{\l_i\e_i}\in\Omega^1_c(P,\fg)_{\rho_\fg}^\hor$ for $(\e_i)_i \in \bigoplus_{i\in \N} \Omega^1(\ol{V}_i,\fg)$. The map $\Psi\co  \bigoplus_{i\in \N} \Omega^1(\ol{V}_i,\fg) \ra \Omega^1_c(P,\fg)_{\rho_\fg}^\hor$, $(\e_i)_i \ms \sum_{i\in\N}\til{\l_i\e_i}$ is continuous, because it is linear and the inclusions $\Omega^1_{\supp(\l_i)}(P,\fg)_{\rho_\fg}^\hor \hra \Omega^1_c(P,\fg)_{\rho_\fg}^\hor$ are continuous.
Let $(\e_i)_i \in \Omega_\oplus$. As in \cite[Proposition 4.6]{Schuett:2013} we get 
\begin{align*}
\Psi((\e_i)_i)_p(w) = \rho_\fg(\b_{\s_{i_0}}(p)\inv ). (\e_{i_0})_{q(p)} (Tq(w))
\end{align*}
if $p \in q\inv(V_{i_0})$ and $w\in T_pP$. In abuse of notation we write $\Phi:=\Phi|^{\Omega_\oplus}$ and $\Psi:=\Psi|_{\Omega_\oplus}$. One easely sees $\Phi\ci\Psi=\id_{\Omega_\oplus}$. It is left to show $\Psi\ci\Phi=\id_{\Omega_c^1(P,\fg)_{\rho_\fg}^\hor}$. Let $\t\in \Omega_c^1(P,\fg)_{\rho_\fg}^\hor$, $p \in P$, $w\in T_pP$ and $i \in\N$ with $p \in q\inv(V_i)$. We calculate
\begin{align}\label{Gru92023}
Tq(w- TR_{\b_{\s_i}(p)} T\s_i Tq(w)) = Tq(w) - T(q\ci R_{\b_{\s_i}(p)} \ci \s_i )(Tq(w)) =0.
\end{align}
Now we get
\begin{align*}
&\Psi\ci\Phi(\t)_p (w) = \rho_\fg(\b_{\s_i}(p)\inv).(\s_i^\ast \t)_{q(p)}(Tq(w))\\
=&\rho_\fg(\b_{\s_i}(p)\inv). \t_{\s_i(q(p))}(T\s_iTq(w))\\
=&\rho_\fg(\b_{\s_i}(p)\inv). \rho_\fg(\b_{\s_i}(p)) \t_{\s_i(q(p))\b_{\s_i}(p)}(TR_{\b_{\s_i}(p)}T\s_iTq(w)) \\
\ub{=}{\ref{Gru92023}}& \t_{\s_i(q(p))\b_{\s_i}(p)}(w) = \t_p(w).
\end{align*}
\end{proof}

\begin{remark}\label{GruPunkteTrennenene}
Let $M$ be a $m$-dimensional manifold, $D\subs TM$ a $d$-dimensional subbundle, $p_0 \in M$ and $w_0 \in D_{p_0}$. Then there exists a smooth curve $\g\co [-1,1] \ra M$ such that $\g(0)=p_0$, $\g'(0)=w_0$ and $\g'(t)\in D_{\g(t)}$ for all $t\in [-1,1]$. In fact  let $\ps\co TU \ra U\ti \R^m$ be a trivialisation with $\ps(D)=U\ti \R^d\ti\set{0}$, $v_0:=\pr_2\ci \ps(w_0) \in \R^d\ti\set{0}$. Then $X\co U \ra TU$, $x\ms \ps\inv(x,v_0)$ is a smooth vector field of $U$ and $\im(X)\subs D$. Let $\til{\g}\co [-\ep,\ep] \ra U$ be the integral curve of $X$ with $\til{\g}(0)=p_0$. Then $\til{\g}'(0)= X(p_0) = w_0$ and obviously $\g'(t) \in D_{\g(t)}$ for all $t$. Now let $\ph \co [-1, 1] \ra [-\ep,\ep]$ be a diffeomorphism with $\ph(0)=0$ and $\ph'(0)=1$. Then $\g:=\til{\g}\ci \ph$ is as needed.
\end{remark}

\begin{lemma}\label{GruPullsepPoints}
The pullbacks $\g^\ast \co \Omega^1(P,\fg)^\hor\ra C^\8([-1,1],\fg)$, $\t \ms \g^\ast \t$ along horizontal maps $\g \co [-1,1]\ra P$ ($\g'(t)\in H_{\g(t)}P$) separate the points in $\Omega^1(P,\fg)^\hor$.
\end{lemma}
\begin{proof}
Let $\t \in \Omega^1(P,\fg)^\hor$ and $\g^\ast\t=0$ for all horizontal curves $\g\co [-1,1] \ra P$. Let $p \in P$ and $w\in T_pP$. We show $\t_p(w)=0$. Because $\t$ is horizontal we can assume $w \in H_pP$. We use Remark \ref{GruPunkteTrennenene} and find a horizontal curve $\g \co [-1,1]\ra P$ with $\g'(0)=w$. Hence $\t_p(w) = \t_p(\g'(0)) = \g^\ast\t(0) =0$.
\end{proof}

One can easily deduce the following easy observation in Remark \ref{GruMussdochnichtsein} from \cite[Theorem 1.11]{Wockel:2007}, but in the special case of a current group on a compact interval the argumentation from \cite[Theorem 1.11]{Wockel:2007} becomes much easier.
\begin{remark}\label{GruMussdochnichtsein}
Let $(U_i)_{i=1,\dots,n}$ be an open cover of the space $[-1,1]$ such that the sets $\ol{U}_i$ are submanifolds with boundary and $G$ be a finite-dimensional Lie group. Then the map $\Phi \co C^\8([-1,1],G) \ra \prod_{i=1}^n C^\8(\ol{U}_i,G)$, $\phi \ms (\phi|_{\ol{U}_i})_i$ is an injective Lie group morphism thats image is a sub Lie group of $\prod_{i=1}^n C^\8(\ol{U}_i,G)$ and $\Phi|^{\im(\Phi)}$ is an isomorphism of Lie groups. We define $\Psi \co C^\8([-1,1], \fg)\ra \prod_{i=1}^nC^\8(\ol{U}_i,\fg)$, $f\ms (f|_{\ol{U}_i})_i$. Let $\exp\co  V_\fg \subs \fg \ra U_G\subs G$ be the exponential function of $G$ restricted to a $0$-neighbourhood such that it is a diffeomorphism. We define the open sets $\mathcal{U}:= C^\8([-1,1],U_G)\subs C^\8([-1,1],G)$ and $\mathcal{V}:=C^\8([-1,1],V_\fg)\subs C^\8([-1,1],\fg)$. Let $\tau_1\co C^\8([-1,1],U_G) \ra C^\8([-1,1],V_\fg)$, $\ph \ms (\exp|_{V_\fg}^{U_G})\inv \ci\ph$ and $\tau_2\co \prod_{i=1}^n C^\8(\ol{U}_i,U_G) \ra \prod_{i=1}^nC^\8(\ol{U}_i,V_\fg)$, $(\ph_i)_i \ms ((\exp|_{V_\fg}^{U_G})\inv \ci\ph_i)_i$ be the canonical charts. We obtain the commutative diagram
\begin{align*}
\begin{xy}\xymatrixcolsep{5pc}
\xymatrix{
C^\8([-1,1],U_G) \ar[r]^-{\Phi|_{\mathcal{U}}} \ar[d]^{\tau_1} & \prod_{i=1}^nC^\8(\ol{U}_i,U_G)\ar[d]^-{\tau_2}\\
C^\8([-1,1],V_\fg)  \ar[r]^-{\Psi|_{\mathcal{V}}}& \prod_{i=1}^nC^\8(\ol{U}_i,V_\fg)
}
\end{xy}
\end{align*}
and calculate
\begin{align*}
&\tau_2\left(\im(\Phi) \cap \prod_{i=1}^nC^\8(\ol{U}_i,U_G)\right) = \tau_2\left( \Phi(C^\8([-1,1],U_G)) \right)\\
=& \Psi(C^\8([-1,1],V_\fg)) = \im(\Psi)\cap \prod_{i=1}^nC^\8(\ol{U}_i,V_\fg)
\end{align*}
The space 
\begin{align*}
\im(\Psi) = \set{(f_i)_i: f_i(x)=f_j(x) ~\mathrm{for}~ x\in \ol{U}_i\cap\ol{U}_j}
\end{align*} 
is closed in $\prod_{i=1}^nC^\8(\ol{U}_i,\fg)$. Hence $\im(\Phi)$ is a sub Lie group of $\prod_{i=1}^nC^\8(\ol{U}_i,G)$. 
In the commutative diagram
\begin{align*}
\begin{xy}\xymatrixcolsep{5pc}
\xymatrix{
C^\8([-1,1],U_G) \ar[r]^-{\Phi|_{\mathcal{U}}} \ar[d]^{\tau_1} & \im(\Phi) \cap \prod_{i=1}^nC^\8(\ol{U}_i,U_G)\ar[d]^-{\tau_2}\\
C^\8([-1,1],V_\fg)  \ar[r]^-{\Psi|_{\mathcal{V}}}& \im(\Psi)\cap \prod_{i=1}^nC^\8(\ol{U}_i,V_\fg)
}
\end{xy}
\end{align*}
the lower vertical arrow is a diffeomorphism because $\Psi \co C^\8([-1,1],\fg) \ra \im(\Psi)$ is a continuous bijective linear map between Fr{\'e}chet spaces. Now the assertion follows.
\end{remark}

The following theorem is a generalisation of \cite[Proposition V.7]{Neeb:2004} in the case of a finite-dimensional codomain.
\begin{theorem}\label{Grunctandk}
\begin{compactenum}
\item The  map $\Delta \co C_c^\8(P,G)_{\rho_G} \ra \Omega_c^1(P,\fg)_{\rho_\fg}^\hor$ is smooth.
\item For the smooth map $\Delta \co C_c^\8(P,G)_{\rho_G} \ra \Omega_c^1(P,\fg)_{\rho_\fg}^\hor$ we have $d_1\Delta(f)=D_{\rho_\fg}f$.\footnote{To show this equation we use Lemma \ref{GruPullsepPoints}. If $H$ is infinite-dimensional, then  we can not apply Lemma \ref{GruPullsepPoints}. Instead on can show that $\d \co C_c^\8(P,G)_{\rho_G} \ra \Omega_c^1(P,\fg)$ is smooth and $d_1\d(f)=df$ (with an argument similar to the argumentation in this proof, because pullbacks by smooth maps (not nessesary horizontal) speperate the points in $\Omega_c^1(P,\fg)$). Obviously $\pr_h^\ast$ is linear. Because $\Delta \co C_c^\8(P,G)_{\rho_G} \ra \Omega_c^1(P,\fg)_{\rho_\fg}^\hor$ is smooth also $\Delta = \pr^\ast\ci \d \co C_c^\8(P,G)_{\rho_G} \ra \Omega_c^1(P,\fg)^\hor$ is smooth. Now one sees that for this corser topology we have $d_1\Delta(f)=\pr_h^\ast(df)$. Then also the derivative with respect to the finer topology on the domain ($\Omega_c^1(P,\fg)^\hor_{\rho_\fg}$) has to be given by this equation. We remaind the reader that  the topology on $\Omega_c^1(P,\fg)^\hor_{\rho_\fg}$ is not the induced topology by $\Omega_c^1(P,\fg)$, but the inductive limit toplogy. Where the inductive limit topology on $\Omega_c^1(P,\fg)^\hor$ and the induced topology realy coincide.}
\end{compactenum}
\end{theorem}
\begin{proof}
\begin{compactenum}
\item Because of Lemma \ref{GruLemmareresk09834t}, Lemma \ref{Gruddkdfjhlemmsdg78} and Lemma \ref{Grukdjllr65rglkdf} it is enough to show the smoothness of $\Delta$ on a $1$-neighbourhood.  Let $(\sigma_i,\ol{V}_i)_{i\in\N}$ be a locally finite compact trivialising system in the sense of \cite[Definition 3.6.]{Schuett:2013} (the existence follows from \cite[Corollary 3.10]{Schuett:2013}).  With the help of Lemma \ref{Grudkdkdlsjlf098} and Lemma \ref{Grukljgdfljhg} it is enough to construct smooth maps $\ps_i\co C^\8(\ol{V}_i, G)\ra \Omega^1(\ol{V}_i,\fg)$ such that  the diagram
\begin{align*}
\begin{xy}\xymatrixcolsep{5pc}
\xymatrix{
C_c^\8(P,G)_{\rho_G} \ar[d]_-{\ph\ms (\ph\ci\sigma_i)_i} \ar[r]^-{\Delta} & \Omega^1_c(P,\fg)_{\rho_\fg}^\hor \ar[d]^-{\t \ms (\sigma_i^\ast\t)_i}\\
\prod^\ast_{i\in\N} C^\8(\ol{V}_i,G) \ar[r]^-{\psi_i}& \bigoplus_{i\in\N} \Omega^1(\ol{V}_i,\fg)
}
\end{xy}
\end{align*}
commutes. Let $\tau_i \co q\inv(V_i)\ra V_i\ti H$, $p \ms (q(p), \ph_i(p))$ be the inverse to $(x,h)\ms \s_i(x)h$. Then $\sigma_i(x)= \tau_i(x,1)$. 
For $f\in C^\8(\ol{V}_i,G)$ we define
\begin{align*}
\til f \co q\inv(V_i)\ra G,~ p \ms \r_G(\ph_i(p) , f(q(p))).
\end{align*}
If $f \in C^\8_c(P,G)_{\r_G}$, then $\wtil{f\ci \s_i} = f|_{q\inv(V_i)}$.
We define $\ps_i\co C^\8(\ol{V}_i, G)\ra \Omega^1(\ol{V}_i,\fg)$ by 
\begin{align*}
\ps_i(f)_x(v) = T\l_{f(x)\inv} T\til{f}(\pr_h(T_x\sigma_i(v))). 
\end{align*}
At first we show that the above diagram commutes. We calculate
\begin{align*}
\ps_i(f\ci\s_i)_x(v) = T\l_{(f_i\ci\s(x))\inv} Tf (\pr_h \ci T_x\s_i(v)) = \sigma_i^\ast(\d(f)\ci \pr_h)_x(v).
\end{align*}
It is left to show the smoothness of $\ps_i$. 
Because we can embed $\Omega^1(\ol{V}_i,\fg)$ into $C^\8(T\ol{V}_i, \fg)$, we show that 
\begin{align*}
C^\8(\ol{V}_i ,G)\ti (TV_i) \ra \fg, ~(f,v)\ms T\l_{f(x)\inv} T\til{f}(\pr_h(T_x\sigma_i(v)))
\end{align*}
is smooth. Let $m\co G\ti G\ra G$ be the multiplication on $G$ and $n\co G\ra TG,~ g\ms 0_g$ the zero section. 
Given $f \in C^\8(\ol{V}_i ,G)$ and $v\in T_x\ol{V}_i$ we calculate
\begin{align*}
\ps_i(f)(v) = Tm(n(f(\p(v))\inv )), T\til{f}(\pr_h T\sigma_i(v)) ).
\end{align*}
The map $\ev\co C^\8(\ol{V}_i,G) \ti \ol{V}_i \ra G$, $f,x\ms f(x)$ is smooth (see \cite[Lemma 121]{Alzaareer:2013}). Therefore it is left to show the smoothness of 
$C^\8(\ol{V}_i,G) \ti Tq \inv (TV_i) \ra TG,~ (f ,v)\ms T\til{f}(v)$ 
The map $\ev^q\co C^\8(\ol{V}_i,G) \ti q\inv(V_i) \ra G, (f,p) \ms f\ci q (p)$ is smooth, because $\ev$ is smooth. We have $T(f\ci q)(v))= T\ev^q(f, \bl)(v) = T\ev^q(n(f),v)$, where $n$ is the zero section of $C^\8(\ol{V}_i,G)$. Hence 
\begin{align*}
T\ev^q \ci (n,\id) \co C^\8(\ol{V},G) \ti Tq\inv(TV) \ra TG,~ (f ,v) \ms Tf \ci Tq (v)
\end{align*}
is smooth.
With $T\til{f} = T\r_G \ci (T\ph_i, Tf\ci Tq)$ the assertion follows from the smoothness of $T\ev^q \ci (n,\id)$.  
\item We write $\d^l\co C^\8([-1,1],G) \ra  C^\8([-1,1],\fg)$ for the classical left logarithmic derivative. It is known that $d_{c_1}\d^l(f) =f'$ for $f\in C([-1,1],\fg)$ (\cite[Lemma 2.4]{Glockner:2012}).
Given a horizontal curve $\g\co [-1,1]\ra P$  we define the maps
\begin{align*}
&\g^\ast_G\co C^\8_c(P,G)_{\r_G} \ra C^\8([-1,1],G), \ph \ms \ph\ci \g\\
&\g^\ast_\fg \co C^\8_c(P,\fg)_{\r_\fg} \ra C^\8([-1,1],\fg), f\ms f\ci \g\tx{and}\\
&\g^\ast_\Omega \co \Omega^1_c(P,\fg)_{\rho_\fg}^\hor \ra C^\8([-1,1],\fg), \t \ms \g^\ast\t.
\end{align*}
Like in Lemma \ref{GruDerPullBackSmooth} one shows that $\g^\ast_G$ is a smooth Lie group homomorphism with $L(\g^\ast_G) = \g_\fg^\ast$ (see Remark \ref{GruMussdochnichtsein}).
The  diagram
\begin{align}\label{GruPullCommd98}
\begin{xy}\xymatrixcolsep{5pc}
\xymatrix{
C^\8_c(P,G)_{\r_G}\ar[r]^-{\Delta} \ar[d]_-{\g_G^\ast} &\Omega_c^1(P,\fg)^\hor_{\rho_\fg} \ar[d]^-{\g_\Omega^\ast}\\
C^\8([-1,1],G) \ar[r]^-{\d^l} & C^\8([-1,1],\fg) 
}
\end{xy}
\end{align}
commutes, because 
\begin{align*}
&(\g_\Omega^\ast \Delta(f))(t)= \d(f) (\pr_h(\g'(t))) = \d(f) (\g'(t)) =T\l_{f\ci\g(t) \inv} \ci Tf(\g'(t)) \\
=& \d^l(f\ci \g).
\end{align*}
Let $f\in C^\8_c(P,\fg)_{\rho_\fg}$. We want to show $d_1\Delta(f) = D_{\rho_\fg}f$. Because of Lemma \ref{GruPullsepPoints} it is enough to show $\g_\Omega^\ast(d_1\Delta(f)) = \g_\Omega^\ast(D_{\rho_\fg}f)$ for an arbitrary horizontal curve $\g \co [-1,1]\ra P$.
Because $\g_\Omega^\ast$ is continuous linear and the diagram (\ref{GruPullCommd98}) commutes we can calculate
\begin{align*}
\g_\Omega^\ast(d_1\Delta(f)) = d_1(\g_\Omega^\ast\ci\Delta)(f) = d_1(\d^l\ci \gamma^\ast_G)(f)= d_1(\d^l)(L(\g^\ast_G)(f)) = (f\ci\g)'.
\end{align*}
Now we use that $\g$ is horizontal and obtain
\begin{align*}
\g_\Omega^\ast(D_{\rho_\fg} f)_t = D_{\rho_\fg}f(\g'(t)) = df(\g'(t)) = \g_\Omega^\ast(d_1\Delta(f))_t
\end{align*}
for $t\in [-1,1]$.
\end{compactenum}
\end{proof}

The proof of the following Lemma \ref{GruMainII} is analogous to the first part of \cite[Proposition III.3]{Maier:2003}.
\begin{lemma}\label{GruMainII}
In the following we write $\Ad$ for the adjoined action of $C_c^\8(P,G)_{\rho_G}$ on $C^\8_c(P,\fg)_{\rho_\fg}$. The map
\begin{align*}
\mathcal{A}\co &C^\8_c(P,G)_{\r_G} \ti (\ol{\Omega}_c^1(P,V)_{\r_V}^\hor \ti_{\o_M} C^\8_c(P,\fg)_{\r_\fg}) \ra \ol{\Omega}_c^1(P,V)_{\r_V}^\hor \ti_{\o_M} C^\8_c(P,\fg)_{\r_\fg}\\
 &(\ph , ([\a],f)) \ms ([\a] - [\k_\fg(\gd(\ph),f)], \Ad(\ph,f)).
\end{align*}
is a smooth group action and its associated Lie algebra action is given by the adoined action described in (\ref{GruAdacttitit234}).
Hence, the adjoint action of $\gg_c(M,\frak{G})$ on the extension $\widehat{\gg_c(M,\frak{G})}:=\ol{\Omega}_c^1(M,\mathbb{V}) \ti_{\o_M} \gg_c(M,\frak{G})$ represented by $\o_M$ integrates to a Lie group action of $\gg_c(M,\mathcal{G})$ on $\widehat{\gg_c(M,\frak{G})}$.
\end{lemma}
\begin{proof}
The smoothness of $\mathcal{A}$ follows from the smoothness of $\Delta$. We show that $\mathcal{A}$ is a group action.
For $\ph,\ps \in C^\8_c(P,G)_{\r_G}$ we have 
\begin{align*}
\gd(\ph \ps) =\gd(\ps) + \Ad^G_\ast(\ps\inv, \gd(\ph)).
\end{align*} 
And for $v,w\in \fg$ and $g\in G$ we have
\begin{align}\label{Grunnnnnnf}
\k_\fg(v,\Ad^G_g w) = \k_\fg(\Ad^G_{g\inv}v,w).
\end{align}
In this context $\Ad^G$ is the adjoint action of $G$ on $\fg$.
Now we calculate for $\a \in \Omega^1_c(P,V)_{\r_V}^\hor$
\begin{align*}
&\mathcal{A}(\ph \. \ps ,([\a],f)) = ([\a]- [\k_\fg(\gd(\ph\ps),f)], \Ad_{\ph\ps}f)\\
=&([\a] - [\k_\fg(\gd\ps , f)] - [\k_\fg (\Ad^G_\ast(\ps\inv, \gd(\ph)), f)] , \Ad_\ph . \Ad_\ps . f) \\
\ub{=}{(\ref{Grunnnnnnf})}&([\a] - [\k_\fg(\gd\ps , f)] - [\k_\fg (\gd(\ph) , \Ad^G_\ast({\ps} f))] , \Ad_\ph . \Ad_\ps . f)\\
=& \mathcal{A}(\ph, ([\a]- [\k_\fg(\gd \ps, f)], \Ad_{\ps}f))\\
=& \mathcal{A}(\ph, (\mathcal{A}(\ps,([\a],f))).
\end{align*}
The associated action to $\mathcal{A}$ by $C^\8_c(P,\fg)_{\rho_\fg}$ is given by the  adoined action described in (\ref{GruAdacttitit234}), because $(-[\k_\fg(D_{\rho_\fg}(g),f)], \ad(g,f)) = ([\k_\fg(g,D_{\rho_\fg}(f)], \ad(g,f))$ for $f,g \in C^\8_c(P,\fg)_{\rho_\fg}$.
\end{proof}
 
\begin{theorem}\label{GruMainIII}
If $\ol{H}$ is finite, then we find a Lie group extension 
\begin{align*}
\ol{\Omega}_c^1(M,\mathbb{V})/\gp_{\o_M} \hra \widehat{\gg_c(M,\mathcal{G})_0} \ra \gg_c(M,\mathcal{G})_0
\end{align*}
that corresponds to the central Lie algebra extension that is represented by $\o_M$.
\end{theorem}
\begin{proof}
We simply need to put Theorem \ref{GruMainI}, Theorem \ref{GruMainII}, \cite[Proposition 7.6]{Neeb:2002} and \cite[Theorem 7.12]{Neeb:2002} together (the last two statements were recalled in the introduction to this theses).
\end{proof}

\section{Universality of the Lie group extension}\label{GruZweiterTeil}
In this section we want to proof \cite[Theorem I.2.]{Janssens:2013} in the case where $M$ is not compact but $\sigma$-compact (like in \cite[Theorem I.2.]{Janssens:2013} $M$ still has to be connected).
In the first part of \cite{Janssens:2013} Janssens and Wockel showed that the cocycle $\o_M\co \gg_c(M,\frak{G})^2 \ra \ol{\Omega}^1_c(M,\mathbb{V})$ (see \cite[p. 129 (1.1)]{Janssens:2013} and Remark \ref{Gru89Remarkdkdf}) is universal if $\fg$ is semisimple and $M$ is a $\sigma$-compact manifold. In the second part of the paper they assumed the base manifold $M$ to be compact and got a universal cocycle $\gg(M,\frak{G})^2 \ra \ol{\Omega}^1(M,\mathbb{V})$. Then they show, that under certain conditions a given Lie group bundle $G\hra \mathcal{G}\ra M$ with finite-dimensional Lie group $G$ is associated to the principal frame bundle $\Aut(G) \hra \text{Fr}(\mathcal{G}) \ra M$. 
Hence they were able to use \cite[Theorem 4.24]{Neeb:2009} to integrate the universal Lie algebra cocycle $\gg(M,\frak{G})^2 \ra \ol{\Omega}^1(M,\mathbb{V})$ to a Lie group cocycle $Z\hra \widehat{\gg(M,\mathcal{G})_0} \ra \gg(M,\mathcal{G})_0$. At this point it was crucial that $M$ was compact and connected in order to apply \cite[Theorem 4.24]{Neeb:2009}. Once the Lie group extension was constructed, Janssens and Wockel proved its universality by using the Recognition Theorem from \cite{Neeb:2002a} (see \cite[Theorem I.2.]{Janssens:2013}).
To generalise \cite[Theorem I.2.]{Janssens:2013} to the case where $M$ is connected and not compact, much of the proofs of \cite{Janssens:2013} can be transfered to case of a non-compact base manifold by using Theorem \ref{GruMainIII} instead of \cite[Theorem 4.24]{Neeb:2009}. But our proof is shorter, because we can use our Theorem \ref{GruMainIII} that holds for section groups and not just for gauges group, while \cite[Theorem 4.24]{Neeb:2009} holds only for gauge groups. Hence we do not have to reduce the statement to the case of gauge groups, like it was done in \cite{Janssens:2013}. We mention that in this section we assume the typical fiber $G$ of the Lie group bundle to be connected, while in \cite{Janssens:2013} Janssens and Wockel assume $\p_0(G)$ to be finitely generated

\begin{convention}
In this section $G$ is a connected semisimple finite-dimensional Lie group. Like in the rest of the paper, $M$ still is a connected, non-compact, $\sigma$-compact finite-dimensional manifold.
\end{convention}

Analogously to \cite[p. 130]{Janssens:2013} we consider the following setting\footnote{In \cite{Janssens:2013} the Lie group $G$ is not assumed to be connected. Instead Janssens and Wockel assume $\p_0(G)$ to be finite generated.}:
\begin{definition}[Cf. p. 130 in \cite{Janssens:2013}]\label{GruUni1}
Let $G$ be a connected finite-dimensional semisimple Lie group with Lie algebra $\fg$ and $G\hra \mathcal{G} \xra{q} M$ be a Lie group bundle. Like in \cite[11.3.1]{Hilgert:2012} we turn $\Aut(G)$ into a finite-dimensional Lie group. 
In particular $\Aut(G)$ becomes a Lie group such that $L\co \Aut(G)\ra \Aut(\fg)$ is an isomorphism onto a closed subgroup (\cite[Lemma 11.3.3]{Hilgert:2012}) and $\Aut(G)$ acts smooth on $G$.
\end{definition}

\begin{lemma}\label{GruUni2}
The Lie group bundle $G\hra \mathcal{G}\xra{q} M$ is isomorphic to the associated Lie group bundle of the frame principal bundle $\Aut(G) \hra \mathrm{Fr}(\mathcal{G}) \ra M$ (cf. \cite[p. 130]{Janssens:2013}). Obviously all manifolds are $\sigma$-compact, because $M$ is $\sigma$-compact and $\Aut(G)$ is homeomorph to an closed subgroup of $\Aut(\fg)$.
\end{lemma}

\begin{definition}
We define $V=V(\fg)$. In the situation considered in this subsection the map $\r_V \co \Aut(G)\ti V \ra V,~ (\ph,\k_\fg(v,w)) \ms \k_\fg(L(\ph)(v),L(\ph)(w))$ is the smooth automorphic action $\rho_V$ described in Convention \ref{GruConvention9898}.
\end{definition}

\begin{lemma}[Cf. p. 130 in \cite{Janssens:2013}]\label{GruActTriv}
The identity component of $\Aut(G)$ acts trivially on  by the representation $\r_V \co \Aut(G)\ti V \ra V,~ (\ph,\k_\fg(v,w)) \ms \k_\fg(L(\ph)(v),L(\ph)(w))$.
\end{lemma}
\begin{proof}
Obviously it is enough to show that $(\Aut(\fg))_0$ acts trivially by $\r \co\Aut(\fg) \ti V \ra V,~(\ph,\k_\fg(x, y))\ms \k_\fg(\ph(x),\ph(y))$. For $\check{\r}\co \Aut(\fg) \ra GL(V),~ \ph \ms \r(\ph,\bl)$, $x,y \in \fg$ and $f\in \der(\fg)$ we have $L(\check{\r})(f)(\k_\fg(x, y)) = d_{\id} \r (\bl,\k_\fg(x, y))(f)$. Defining $\ev_x \co \Aut(\fg) \ra \fg,~ \ph \ra \ph(x)$ for $x\in \fg$ we get 
\begin{align*}
\r(\bl,\k_\fg(x, y)) = \k_\fg\ci (\ev_x,\ev_y).
\end{align*}
We have $d_{\id} \ev_x(f) = \fr{\partial}{\partial t}|_{t=0} \exp(tf)(x) = f(x)$. Hence
\begin{align*}
&d_{\id} \r(\bl,\k_\fg(x,y)) (f) = \k_\fg(\ev_x(\id), d_{\id} \ev_y(f)) + \k_\fg(d_{\id} \ev_x(f) , \ev_y(\id))\\
=&\k_\fg(x,f(y)) + \k_\fg(f(x),y).
\end{align*}
Because $\fg$ is semisimple, $\der(\fg) = \mathrm{inn}(\fg)$. For $z\in \fg$ we calculate
\begin{align*}
L(\check{\r})(\ad_z) (\k_\fg(x,y)) = \k_\fg(x,[y,z]) + \k_\fg([x,z],y) = \k_\fg(x,[y,z]) + \k_\fg(x,[z,y]) =0.
\end{align*}
Hence $\check{\r}|_{\Aut(\fg)_0} = \id_V$.
\end{proof}

Analogously to \cite[p. 130]{Janssens:2013}, we need the following requirement:
\begin{convention}\label{GruConventioooo}
In the following we assume $\ol{\Aut(G)}:= \Aut(G)/\ker(\r_V)$ to be finite.
\end{convention}

\begin{definition}
Combining Convention \ref{GruConventioooo}, Lemma \ref{GruActTriv} and Theorem \ref{GruMainIII} we find a Lie group extension
\begin{align*}
\ol{\Omega}_c^1(M,\mathbb{V})/\gp_{\o_M} \hra \widehat{\gg_c(M,\mathcal{G})_0} \ra \gg_c(M,\mathcal{G})_0
\end{align*}
that corresponds to the central Lie algebra extension that is represented by $\o_M$.  
We write $Z:=\ol{\Omega}_c^1(M,\mathbb{V})/\gp_{\o_M}$. If $\p \co \widetilde{\gg_c(M,\mathcal{G})_0} \ra \gg_c(M,\mathcal{G})_0$ is the universal covering homomorphism and $Z\hra H \ra \widetilde{\gg_c(M,\mathcal{G})_0}$ the pullback extension then \cite[Remark 7.14.]{Neeb:2002} tells us that we get a central extension of Lie groups
\begin{align*}
E:= Z\ti \p_1(\gg_c(M,\mathcal{G})_0) \hra H \ra \gg_c(M,\mathcal{G})_0.
\end{align*}
Its corresponding Lie algebra extension is represented by ${\o_M}$.
\end{definition}

The following theorem is the analogous statement to \cite[Theorem I.2.]{Janssens:2013} in the case of a non-compact base-manifold and  connected typical fibre.
\begin{theorem}
The central Lie group extension $Z\ti \p_1(\gg_c(M,\mathcal{G})_0) \hra H \ra \gg_c(M,\mathcal{G})_0$ is universal for all abelian Lie  groups modelled over locally convex spaces.
\end{theorem}
\begin{proof}
The statement \cite[Theorem 4.13]{Neeb:2002a} and the analogous statement \cite[Theorem III.1]{Janssens:2013} are formulated for sequentially complete respectively Mackey complete spaces. But the completeness is only assumed to guaranty the existence of the period map $\per_\o$ and the existence of period maps of the form $\per_{\g\ci \o}$ for  continuous linear maps $\g \co \frak{z}\ra \frak{a}$. Obviously the period maps $\per_{\g\ci \o}$ exist if the period map $\per_\o$ exists.  
Hence with Remark \ref{GruVollststaendigno} we do not need to assume the completeness of the spaces. Therefore it is left to show that $H$ is simply connected. Using \cite[Remark 5.12]{Neeb:2002} we have the long exact homotopy sequence
\begin{align*}
&\p_2(\gg_c(M,\mathcal{G})_0) \xra{\d_2} \p_1(Z\ti \p_1(\gg_c(M,\mathcal{G})_0)) \xra{i} \p_1(H) \xra{p} \p_1(\gg_c(M,\mathcal{G})_0) \\
&\xra{\d_1} \p_0(Z\ti \p_1(\gg_c(M,\mathcal{G})_0)).
\end{align*}
We show $i=0$: Calculating
\begin{align*}
\p_1(Z\ti \p_1(\gg_c(M,\mathcal{G})_0))  = \p_1(\ol{\Omega}^1_c(M,\mathbb{V})/\gp_{\o_M}) = \gp_{\o_M}
\end{align*}
and using \cite[Proposition 5.11]{Neeb:2002} we conclude that $\d_2$ is surjective. Hence $i=0$. From
\begin{align*}
\p_0(Z\ti \p_1(\gg_c(M,\mathcal{G})_0))= \p_1(\gg_c(M,\mathcal{G})_0),
\end{align*}
we get that $\d_1$ is injective. Therefore $p=0$. Thus $\p_1(H)=0$.
\end{proof}

\appendix
\section{Some differential topology}\label{GruAppendix}
\begin{lemma}\label{GruRealisation}
Let $H\hra P\xra{q} M$ be a finite-dimensional smooth principal bundle (with $\sigma$-compact total space $P$), $\rho\co H \ti V\ra V$ a finite-dimensional smooth linear representation and $\mathbb{V}:=P\ti_{\rho}V$ the associated vectorbundle.
\begin{compactenum}
\item The canonical isomorphism of vector spaces (see e.g. \cite[Satz 3.5]{Baum:2014})
$\gph \co \Omega^k(P,V)_\r^\hor \ra \Omega^k(M,\mathbb{V})$, $\o\ms \til{\o}$ (with $\til{\o}_x(v_1,\dots,v_k)=\o_{\s_(x)}(T\s(v_1),\dots, T\s(v_k))$ for a local section $(\s\co U \ra P$ of $P\xra{q}M$ and $x\in U$)
is in fact an isomorphism of topological vector spaces.
\item Also the isomorphism of vector spaces
$\gph \co \Omega^k_c(P,V)_\r^\hor \ra \Omega_c^k(M,\mathbb{V})$, $\o\ms\til{\o}$
is in  an isomorphism of topological vector spaces.
\end{compactenum}
\end{lemma}
\begin{proof}
\begin{compactenum}
\item We choose an atlas $\ps_i \co q\inv(U_i)\ra U\ti H$ of trivialisations of $P$ with $i\in I$. Let $\sigma_i:= \ps_i\inv(\bl,1_H)$ be the canonical section corresponding to $\ps_i$. As $\Omega^k(P,V)_\r^\hor$ and $\Omega^k(M,\mathbb{V})$ are Fr{\'e}chet spaces it is enough to show the continuity of $\gph$ (Open mapping theorem). The topology on  $\Omega^k(M,\mathbb{V}) = \Gamma(\Alt^k(TM,\mathbb{V}))$ is initial with respect to the maps  $\Gamma(\Alt^k(TM,\mathbb{V}))\ra  \Gamma(\Alt^k(TU_i,\mathbb{V}|_{U_i}))$, $\eta \ms \eta|_{U_i}$. 
Given $\o\in \Omega^k(P,V)_\r^\hor$, $x\in U_i$ and $v\in T_xU_i$ we have $(\til{\o}|_{U_i})_x(v) =[\sigma_i(x),\sigma_i^\ast\o_x(v)]$. Because $\Gamma(\Alt^k(TU_i,\mathbb{V}|_{U_i}))\cong \Gamma(\Alt^k(TU_i,V)) \cong \Omega^k(U_i,V)$ it is enough to show the continuity of 
$\Omega^k(P,V)_{\rho_V}^\hor \ra \Omega^k(U_i,V)$, $\o \ms \sigma_i^\ast \o$. 
The map $C^\8((TP)^k,V)\ra C^\8((TU_i)^k,V)$, $f\ms f\ci T\sigma_i\ti \dots\ti T\sigma_i$ is continuous (see \cite{Glockner:a}). Now the assertion follows, because we can embed $\Omega^k(P,V)_{\rho_V}^\hor$ into $C^\8((TP)^k,V)$.
\item The analogous map from $\Omega^k(P,V)_\r^\hor$ to $\Omega^k(M,\mathbb{V})$ is continuous. Hence given a compact set $K \subs M$ we get that the corresponding map from $\Omega^k_K(P,V)_\r^\hor$ to $\Omega_K^k(M,\mathbb{V})$ is continuous. Therefore $\gph$ is continuous. The same argument shows that the inverse of $\gph$ is continuous.
\end{compactenum}

\end{proof}

The basic consideration in the following remark, seems to be part of the folklore.
\begin{remark}\label{GruIsoQuotientBundle}
Given the situation of Definition \ref{GruQuotientenBuendel} the following holds.
\begin{compactenum}
\item The vertical bundle of $\ol{H} \hra \ol{P} \xra{\ol{q}} M$ is given by
$V\ol{P} = T\p (VP))$ and $H\ol{P}:= T\p(HP)$ is a principal connection on $\ol{P}$.
\- Given $k \in \N_0$ the pullback $\p^\ast \co \Omega^k(\ol{P},V)_{\ol{\r}_V}^\hor \ra \Omega^k(P,V)_{\r_V}^\hor,~ \t \ms \p^\ast \t$ is an isomorphism of topological vector spaces and an isomorphism of chain complexes.
\- Given $k \in \N_0$ the pullback $\p^\ast \co \Omega_c^k(\ol{P},V)_{\ol{\r}_V}^\hor \ra \Omega_c^k(P,V)_{\r_V}^\hor,~ \t \ms \p^\ast\t$ is an isomorphism of topological vector spaces and an isomorphism of chain complexes.
\end{compactenum}
\end{remark}
\begin{proof}
\begin{compactenum}
\item First we show $T\p(VP) \subs \ker(T\ol{q})$. For $v\in VP$ we get $T\ol{q} (T \p(v))= T \ol{q} \ci \p (v) = T q (v) =0$.  To see $\ker(T\ol{q}) \subs T\p(VP)$ let $T\ol{q} (w) = 0$. We find $v\in TP$ with $T\p (v) = w$. Hence $T\ol{q}(T\p(v)) = T\ol{q} \ci \p  (v) = Tq(v) = 0$. Thus $v\in VP$ and so $w\in T\p(VP)$.
Now we show that $T\p(HP)$ is a smooth sub vector bundle of $T\ol{P}$. Let $\ol{x}\in \ol{P}$.  Obviously $(T\p(HP))_{\ol{x}} := T_{\ol{x}}\ol{P} \cap T\pi(HP)$ is closed under scalar multiplication. Let $v,w \in (T\p(HP))_{\ol{x}} = T_x\ol{P}\cap T\p(HP)$. We find $p_1,p_2 \in P$, ${v}_1\in H_{p_1}P$ and $w_2 \in H_{p_2}P$ with $T_{p_1}\pi (v_1) = v$ and $T_{p_2}\pi (w_2) = w$. Hence $\pi(p_1)=\ol{x}=\pi(p_2)$. Therefore we find $n \in N$ with $p_1=p_2\cdot n$ and $\til{w}\in T_{p_1}P$ with $TR_n(\til{w}) = w_2$. Now we calculate 
\begin{align*}
&v+w = T_{p_1}\pi(v_1) + T_{p_2} \pi (w_2) = T_{p_1}\pi(v_1) + T_{p_2} \pi \ci T_{p_1}R_n (\til{w}) \\
= &T_{p_1}\pi(v_1) + T_{p_1} \pi \ci TR_n (\til{w}) = T_{p_1}\pi(v_1) + T_{p_1} \pi (\til{w}) = T_{p_1}\pi (v_1 +\til{w}).
\end{align*}
Now we can show that $HP$ is a smooth sub vector bundle. Let $\ol{p} \in \ol{P}$. Because $\pi$ is a submersion, we find a smooth local section $\ta \til{V} \ra P$ of $\pi$ on an open $\ol{p}$-neighbourhood $\til{V}\subs \ol{P}$. We define $p:=\tau (\ol{p})$ and find a smooth local frame $\sigma_1, \dots ,\sigma_m \co \til{U} \ra TP$ of the smooth sub vector bundle $HP$ on a $p$-neighbourhood $\til{U}\subs P$. Without loose of generality we can assume $\tau (\til{V}) \subs \til{U}$. Given $i \in \set{1,\dots, m}$ we define the smooth map
\begin{align*}
\ol{\sigma}_i \co \til{V} \ra T\ol{P},~ \ol{x}\ms T\pi(\sigma_i \ci \tau (\ol{x})).
\end{align*}
The map $\ol{\sigma}_i$ is a section for the tangential bundle $T\ol{P}$, because given $\ol{x}\in \til{V}$ we have $\sigma_i \ci \tau (\ol{x}) \in T_{\tau(\ol{x})}P$ and so $\ol{\sigma}_i (\ol{x}) \in T_{\pi(\tau(\ol{x}))} \ol{P} = T_{\ol{x}} \ol{P}$. Let $\ol{x}\in \til{V}$. Next we show  that $(\sigma_i(\ol{x}))_{i=1,\dots,m}$ is a basis of $(T\p(HP))_{\ol{x}} = T_{\ol{x}} \ol{P}\cap T\p(HP)$. Let $\l_i \in \R$ with $\sum_{i=1}^m \l_i \cdot \ol{\sigma}(\ol{x}) =0$. We conclude $T_{\tau(\ol{x})} \pi (\sum_{i=1}^m \l_i \cdot {\sigma}_i(\tau( \ol{x})))=0$. And so 
\begin{align*}
Tq\left(\sum_{i=1}^m \l_i \cdot {\sigma}_i(\tau( \ol{x}))\right) = T\ol{q} \left(T\pi\left(\sum_{i=1}^m \l_i \cdot {\sigma}_i(\tau( \ol{x}))\right)\right) = 0.
\end{align*}
Therefore $\sum_{i=1}^m \l_i \cdot {\sigma}_i(\tau( \ol{x})) \in V_{\tau(\ol{x})}P$. And hence $\l_i =0$ for $i=1,\dots,m$. Let $p \in P$ with $\pi(p)=\ol{x}$. One easily sees that the linear map $(T_p\pi)|_{H_pP} \co H_pP \ra (T\p(HP))_x$ is a surjection (see above). Because $m = \dim(H_pP)$ the linear independent system $\ol{\sigma}_i(\ol{x})_{i=1, \dots,m}$ is a basis of $(T\pi(HP))_{\ol{x}}$. Now lets show that $H\ol{P}:= T\pi(HP)$ is a principal connection on $\ol{P}$. Because $\p$ is a submersion and $T_pP = H_pP \oplus V_pP$ we get $V_{\ol{x}}\ol{P} + H_{\ol{x}}\ol{P} = T_{\ol{x}} \ol{P}$ for $\ol{x}\in \ol{P}$. If $T_{p}\p(v)=T_{p'}\p(w)$ with $v\in V_pP$, $w\in H_{p'}P$ and $\p(p)=\p(p')=:\ol{p}$ we get
\begin{align*}
T_{\ol{p}}\ol{q}\ci \p(v) = T_{\ol{p}}\ol{q} \ci \p(w). 
\end{align*}
Hence $0=Tq(v)=Tq(w)$. Thus $w \in V_{p'}P$. Therefore $w =0$ and so $T_p\p(v)=T_{p'}(w)=0$ in $T_{\ol{p}}\ol{P}$. We conclude $V\ol{P} \oplus H\ol{P} = T\ol{P}$. It is left to show that $H\ol{P}$ is invariant under the action of $\ol{H}$. Obviously it is enough to show $T_{\ol{x}}\ol{R}_{[g]} (H_{\ol{x}}P) \subs H_{\ol{x}[g]}\ol{P}$ for $\ol{x} \in \ol{P}$ and $[g]\in \ol{H}$. Let $v\in H_{\ol{x}}\ol{P}$. We find $p \in P$ and $w \in H_pP$ with $v= T_p\pi(w)$. With $\ol{R}_{[g]}\ci \pi = \pi \ci R_g$ and $\pi(pg)= \ol{x}.[g]$ we calculate
\begin{align*}
&T\ol{R}_{[g]}(v) = T_p(\ol{R}_{[g]} \ci \pi)(w) =  T_{pg}\pi \ci T_pR_g(w) \in T_{pg} \pi (H_{pg}P)\\ 
\subs& T_{\ol{x}[g]}\ol{P}\cap T\pi(HP) = H_{\ol{x}[g]} \ol{P}. 
\end{align*}
\- First we show that $\p^\ast$ makes sense. Without loss of generality we assume $k =1$. Let $\t \in \Omega^1(\ol{P},V)_{\ol{\r}_V}^\hor$. We have $\p \ci R_g = \ol{R}_{[g]} \ci \p$. Hence 
\begin{align*}
&\r_V(g) \ci R_g^\ast \p^\ast \t = \ol{\r}_V([g]) \ci (\p\ci R_g)^\ast \t = \ol{\r}_V([g]) \ci (\ol{R}_{[g]} \ci \p)^\ast \t\\
=&\p^\ast(\ol{\r}_V([g]) \ci \ol{R}_{[g]} ^\ast \t) = \p^\ast \t.
\end{align*}
Moreover  if $v\in V_pP$ we get $T_p\p(v) \in V_{\p(p)}\ol{P}$ and so $\p^\ast \t_p(v) = \t_{\p(p)} (T_p\p(v)) =0$.
We show that $\p^\ast$ is bijective. 
It is clear that $\p^\ast$ is injective, because $\pi$ is a submersion. To see that $\p^\ast$ is surjective let $\e \in \Omega^1(P,V)_{\r_V}^\hor$, we define $\t \in \Omega^1(\ol{P},V)^\hor_{\ol{\r}_V}$ by $\t_{\p(p)}(T_p\p(v)) := \e_p(v)$ for $p \in P$ and $v\in T_pP$. To see that this is well-defined, we choose $p,r\in P$, $v\in T_pP$ and $w\in T_rP$ with $\p(p)=\p(r)$ and $T_p\p(v) = T_r\p(w)$. We find $ n\in N$ with $p = r.n$. Because $\e_{r.n} (T_rR_n(w)) = \e_r(w)$ ($N=\ker(\r_V)$), it is enough to show $\e_p(v) = \e_p(T_rR_n(w))$. We have $\p \ci R_n= R_{[n]}\ci\p =  \p$. Hence $T\p \ci TR_n = T\p$. Thus $T_p\p(T_rR_n(w)) = T_r\p(w) = T_p\p(v)$. Therefore we find $x \in \ker(T_p\p)$ with $T_rR_n(w)+x =v$ in $T_pP$. Hence $T_pq(x) = 0$, because $Tq = T\ol{q} \ci T\p$.  So $x \in V_pP$ and hence $\e_p(x)=0$. The form $\t$ is $\ol{\r}_V$-invariant because for $g\in H$, $p \in P$ and $v\in T_pP$ we get
\begin{align*}
&(\ol{\r}_V([g])\ci \ol{R}_{[g]}^\ast \t)_{\p(p)} (T_p\p(v)) = \ol{\r}_V([g])\ci \t_{\p(p).[g]} (T\ol{R}_{[g]}(T_p\p(v)))\\
=& \r_V(g) \ci \t_{\p(p.g)}(T_{p.g}\p(TR_g(v)))\\
=&\r_V(g)\ci \e_{p.g}(TR_g(v)) = \t_{\p(p)}(T\p(v)).
\end{align*}
Moreover $\t$ is  horizontal because given $u \in V_{\ol{p}}\ol{P}$ with $\ol{p}\in \ol{P}$ we find $p \in P$ with $\p(p)= \ol{p}$ and $v\in V_pP$ with $u = T_p\p(v)$. Hence $\t_{\ol{p}}(u)= \t_{\p(p)}(T\p(v)) = \e_p(v)=0$. Obviously we have $\p^\ast \t =\e$.  
In order to show that $\p^\ast$ is an isomorphism of chain complexes we choose $p\in P$, and $v,w\in T_pP$ and calculate
\begin{align*}
&(\p^\ast D_{\ol{\r}_V}\t)_p(v,w) = (D_{\ol{\r}_V}\t)_{\p(p)}(T\p(v),T\p(w))\\
=& (d\t)_{\p(p)} (\pr_h\ci T\p (v),\pr_h\ci T\p (w)) = (d\t)_{\p(p)} (T\p \ci \pr_h(v),T\p \ci \pr_h(w))\\
=&(\p^\ast d\t)_p(\pr_h(v), \pr_h(w)) =(D_{\r_V}\p^\ast \t)_p(v,w).
\end{align*}
It is left to show that $\p^\ast$ is a homeomorphism. Because the corresponding spaces are Fr{\'e}chet-spaces it is enough to show the continuity of $\p^\ast$. We can embed $\Omega^k(\ol{P},V)_{\ol{\rho}_V}^\hor$ into $C^\8(T\ol{P}^k,V)$ and $\Omega^k(P,V)_{\rho_V}^\hor$ into $C^\8(TP^k,V)$. The map $C^\8(T\ol{P}^k,V) \ra C^\8(TP^k,V)$, $f\ms f\ci (T\pi \ti \cdots \ti T\pi)$ is continuous (see \cite{Glockner:a}). Now the assertion follows.
\- This follows from (b) and the fact that $\p^\ast(\Omega_K^k(\ol{P},V)_{\ol{\r}_V}^\hor) = \Omega^k_K(P,V)_{\r_V}^\hor$ for a compact set $K\subs M$.
\end{compactenum}
\end{proof}

Also the statement in the following lemma seems to be well-known, but we did not find a source for this exact result. It's proof uses techniques from the proof of \cite[Theorem 1.5]{Rosenberg:1997}. See also \cite[Chapter 6]{Bott:1982}.

\begin{lemma}\label{GruIsoEndlCover}
If $q \co \hat{M} \ra M$ is a smooth finite manifold covering, then $q^\ast\co \Omega^1_c(M,V) \ra \Omega^1_c(\hat{M}, V),~ \t \ms q^\ast \t$ leads to a well-defined isomorphism of topological vector spaces $H^1_{dR,c}(M,V) \ra H^1_{dR,c}(\hat{M},V),~[\t] \ms [q^\ast \t]$. Therefore $\ol{q}^\ast\co H^1_{dR,c}(M,V) \ra H^1_{dR,c}(\ol{P}, V),~ \t \ms \ol{q}^\ast \t$ is an isomorphism of topological vector spaces.
\end{lemma}
\begin{proof}
We use the notation $\Omega_K^k(\hat{M},V):= \set{\t \in \Omega^k(\hat{M},V)| \supp(\t)\subs q\inv(K)}$ for a compact subset $K\subs M$. Let $n$ be the order of the covering. 
The first step is to define a continuous linear map $q_\ast \co \Omega_c^k(\hat{M},V) \ra  \Omega_c^k(M,V)$ for $k \in \N_0$. Without loss of generality let $k=1$. Let $\t \in \Omega^1_c(\hat{M},V)$. Given $y\in M$ we find a $y$-neighbourhood $V_y \subs M$ that is evenly  covered by open sets $U_{y,i}\subs \hat{M}$ with $i=1,\:,n$. We have diffeomorphisms $q_i^y:=q|_{U_{y,i}}^{V_y}$. Then 
\begin{align*}
\til{\t}^y:= \fr 1n \sum_{i=1}^n (q_i^y)_\ast \t|_{U_{y,i}}
\end{align*}
is a form on $V_y$ with $(q_i^y)_\ast \t|_{U_{y,i}} = \t(T{q_i^y}\inv(v))$ for $x\in V_y$ and $v \in T_xV_y$. We define $q_\ast \t:= \til{\t} \in \Omega^1_c(\hat{M},V)$ by $\til{\t}_x:=\til{\t}^y_x$ for $x\in V_y$. Now we show that this is a well-defined map. Let $x\in V_y \cap V_{y'}$ for $y' \in M$ with a $y'$-neighbourhood $V_{y'}$ that is evenly  covered by $(U_{y',i})_{i=1,..,n}$. After renumbering the sets $U_{y',i}$ we get
\begin{align*}
q|_{U_{y,i}} \inv  = q|_{U_{y',i}}\inv
\end{align*}
on $V_y \cap V_{y'}$ for $i=1,\:,n$. Hence 
\begin{align*}
\til{\t}^y_x = \fr 1n \sum_i ((q_i^y)_\ast \t|_{U_{y,i}})_x = \fr 1n \sum_i ((q_i^{y'})_\ast \t|_{U_{y',i}})_x = \til{\t}^{y'}_x \tx{for} x\in V_{y}\cap V_{y'}.
\end{align*}
We note that $q$ is a proper map, because it is a finite covering. Let $\supp (\t)\subs q\inv(K)$ for a compact set $K \subs M$. If $y \notin  K$ then $q\inv(\set{y}) \cap q\inv(K) = \emptyset$. Hence $q\inv(\set{y}) \cap \supp(\t) =\emptyset$. It follows 
\begin{align*}
q_\ast \t_y = \til{\t}_y^y = \fr 1n \sum_i ((q|_{U_i^y})_\ast \t|_{U_{y,i}})_y = \fr 1n \sum_i \t_{q|_{U_i}\inv (y)} =0.
\end{align*}
Hence $M\setminus K \subs M\setminus \set{x\in M: q_\ast \t_x\neq 0}$. Therefore $\set{x\in M: q_\ast \t_x \neq 0} \subs K$ and so $\supp(q_\ast \t) \subs K$.  
Obviously $q_\ast$ is linear. Moreover $q_\ast$ is continuous because the analogous map form $\Omega^1(\hat{M},V)$ to $\Omega^1(M,V)$ is continuous and $q_\ast(\Omega_K^1(\hat{M}, V)) \subs \Omega_K^1(M,V)$.
Moreover $q_\ast$ is a homomorphism of chain complexes: Given $y\in M$, $v,w \in T_yM$ we calculate
\begin{align*}
(q_\ast d\t)_y(v,w) = \fr 1n \sum_i ((q_i^y)_\ast d\t|_{U_{y,i}})_y(v,w) = \fr 1n \sum_i ( d (q_i^y)_\ast \t|_{U_{y,i}})_y(v,w) = (d q_\ast \t)_y(v,w).
\end{align*} 
Now we show 
\begin{align}\label{Grusur}
q_\ast \ci q^\ast = \id_{\Omega^1_c(M,V)}.
\end{align}
Given $\t \in \Omega_c^1(M,V)$, $y \in M$ and $v\in T_yM$ we calculate
\begin{align*}
&(q_\ast q^\ast \t)_y(v) =\fr 1n \sum_i ({q_i^y}_\ast q^\ast \t|_{U_{y,i}})_y(v) = \fr 1n \sum_i (q^\ast \t|_{U_{y,i}})_{{q_i^y}\inv (y)} (T{q_i^y}\inv (v))\\
=&\fr 1n \sum_i  \t_{q({q_i^y}\inv (y))} (Tq\ci{q_i^y}\inv (v)) = \t_y(v).
\end{align*}
Hence $q_\ast \ci q^\ast = \id_{\Omega^1_c(M,V)}$.
We know that $q^\ast$ factorises to a continuous linear map $q^\ast \co H^1_{dR,c}(M,V)\ra H^1_{dR,c}(\hat{M},V)$ and because $q_\ast$ is a homomorphism of chain complexes we get a map  $q_\ast \co H^1_{dR,c}(\hat{M},V) \ra H^1_{dR,c}(M,V)$.
With equation (\ref{Grusur}) we see 
\begin{align*}
q_\ast \ci q^\ast = \id_{H^1_{dR,c}(M,V)}.
\end{align*}
Hence $q_\ast$ is surjective. It remains to show that $q_\ast \co H^1_{dR,c}(M,V) \ra H^1_{dR,c}(\hat{M},V)$ is also injective. To this end we show $q_\ast (B^1_c(\hat{M},V)) = B^1_c(M,V)$. Given $f\in C^\8_c(M,V)$ we calculate
\begin{align*}
q_\ast (d q^\ast f) =q_\ast  q^\ast df =df.
\end{align*}
\end{proof}

The proof of Lemma \ref{GrugeschlForm} is similar to the proof of \cite[Lemma II.10 (1)]{Neeb:2004}.
\begin{lemma}\label{GrugeschlForm}
Let $M$ be a connected finite-dimensional manifold, $E$ be a finite-dimensional vector space and $\t \in \Omega^1(M,E)$. If  for all closed smooth curves $\a_0,\a_1 \co [0,1]\ra M$ such that $\a_0$ is homotopy to $\a_1$ relative $\set{0,1}$  we get
\begin{align*}
\int_{\a_0} \t = \int_{\a_1} \t,
\end{align*}
then $\t \in Z^1_{dR}(M,E)$.
\end{lemma}
\begin{proof}
Let $q\co \til{M} \ra M$ be the universal smooth covering of $M$. 
First we show that $q^\ast \t$ is exact. To this end we show that $q^\ast \t$ is conservative. Let $\g\co [0,1] \ra \til{M}$ be a smooth closed curve in a point $p_0 \in \til{M}$ and $q(p_0)=:x_0 \in M$. Because $\til{M}$ is simply connected, we find  a homotopy $H$ from $\g$ to $c_{p_0}$ relative $\set{0,1}$. Hence $q\ci \g$ is homotopy to $c_{x_0} = q\ci c_{p_0}$ relative $\set{0,1}$. Therefore we get
\begin{align*}
\int_\g q^\ast \t = \int_{[0,1]} \g^\ast q^\ast \t = \int_{[0,1]} (q\ci \g)^\ast \t \ub{=}{(\ast)} \int_{[0,1]} c_{x_0}^\ast \t =0.
\end{align*}
Equation $(\ast)$ follows from the assumptions of the lemma.
Because $q^\ast \t$ is exact we find  $f\in C^\8(\til{M},E)$ with $q^\ast \t =df$. Hence we get
\begin{align*}
q^\ast d\t = dq^\ast \t = ddf =0.
\end{align*}
Therefore $d\t =0$ because $q$ is a submersion.
\end{proof}

\begin{definition}\label{GrustrictSES}
Let $M$ be a connected $\sigma$-compact finite-dimensional manifold. Using \cite[Lemma IV.4]{Neeb:2004} we find a sequence $(K_n)_{n\in \N}$ of compact equidimensional submanifolds with boundaries of $M$ such that $K_n\subs \mathring{K}_{n+1}$, $\bigcup_{n\in \N} K_n = M$ and the connected components of $M\setminus K_n$ are not relative compact in $M$. We call such a sequence $(K_n)_{n\in \N}$ a {\it saturated exhaustive sequence}. The sequence is called {\it strict} if there exists $N\in \N$ such that for all $n\geq N$ the canonical map $\p_0(M\setminus K_{n+1}) \ra \p_0(M\setminus K_n)$ is injective.
\end{definition}

\begin{lemma}\label{GruFolgenvoll}
Let $M$ be a $\sigma$-compact finite-dimensional manifold, $V$ a finite-dimensional vector space and $(K_n)_{n\in \N}$ a strict saturated exhaustive sequence for $M$. Then the space $\go_c^1(M,V)/dC_c^\8(M,V)$ is complete.
\end{lemma}
\begin{proof}
The space $\go_c^1(M,V)$ is the strict inductive limit of the Frechet spaces $(\go^1_{K_n}{M,V})_{n\geq N}$ where $N$ is chosen like in Definition \ref{GrustrictSES}. Using \cite[Lemma B.4]{Neeb:2004} and \cite[Lemma IV.10]{Neeb:2004} we get
\begin{align*}
&\ol{\go}_c^1(M,V):= \go_c^1(M,V)/dC_c^\8(M,V) = \underset{\ra}{\lim}\go^1_{K_n}{M,V} / (dC^\8_c(M,V)\cap \go^1_{K_n}{M,V})\\
=& \underset{\ra}{\lim}\go_{K_n}^1(M,V)/dC_{K_n}^\8(M,V).
\end{align*}
Because the spaces $\go_{K_n}^1(M,V)/dC_{K_n}^\8(M,V)=:\ol{\go}_{K_n}^1(M,V)$ are Frechet spaces, it is enough to show that the inductive limit is strict. Therefore  we have to show that the image 
$\go_{K_n}(M,V)+ dC^\8_{K_{n+1}}(M,V)/ dC^\8_{K_{n+1}}(M,V)$
of the continuous linear injection
%\begin{align*}
$\ol{\go}^1_{K_n} \ra \ol{\go}^1_{K_{n+1}},~ [\t]\ms [\t]$
%\end{align*}
is closed in $\ol{\go}^1_{K_{n+1}}$. For $\a \in C^\8(\mathbb{S}^1,M)$ we define the continuous linear map 
\begin{align*}
I_\a \co \ol{\go}_{K_{n+1}}(M,V) \ra V,~ [\t]\ms \int_\a \t.
\end{align*}
We show
\begin{align*}
\ol{\go}_{K_n}^1(M,V) \cong \go_{K_n}(M,V)+ dC^\8_{K_{n+1}}(M,V)/ dC^\8_{K_{n+1}}(M,V)=  \bigcap_{\a \in C^\8(\mathbb{S}^1, M\setminus \mathring{K}_n)} I_\a \inv (\set{0}).
\end{align*}
One inclusion is trivial. We mention that $M\setminus \mathring{K}_n$ is a submanifold with boundary of $M$. Now let $[\t] \in \ol{\go}^1_{K_{n+1}}(M,V)$ and $\int_\a \t|_{M\setminus \mathring{K}_n} =0$ for all $\a \in C^\8(\mathbb{S}^1,M\setminus \mathring{K}_n)$. Hence $\t|_{M\setminus \mathring{K}_n}$ is conservative and we find $f\in C^\8(M\setminus \mathring{K}_n, V)$ with $\t|_{M\setminus \mathring{K}_n} = df$. Because $\supp(\t|_{M\setminus \mathring{K}_{n}}) \subseteq K_{n+1}$ we get $df|_{M\setminus K_{n+1}}=0$. Hence $f$ is constant on each connected component of $M\setminus K_{n+1}$. Because the saturated exhaustive sequence $(K_n)_{n\in \N}$ is strict we can subtract the constant value of $f$ on each connected component and can assume $\supp(f)\subs K_{n+1}$. Now we extend $f$ to a smooth map $\til{f} \co M\ra V$, what is possible because $M\setminus \mathring{K_n}$ is a closed submanifold. Obviously $\supp (\t - d\til{f}) \subs K_n$ and hence $[\t]= [\t-d\til{f}] \in \ol{\go}_{K_n}(M,V) \subs \ol{\go}_{K_{n+1}}(M,V)$.
\end{proof}

\begin{lemma}\label{GruFolgenvoll2}
In the situation of Section \ref{GruChLieEx}, the space $\go_c^1(M,\mathbb{V})/d_{\mathbb{V}}\gg_c(M,\mathbb{V})$ is sequentially complete if $\ol{P}$ admits a strict saturated exhaustive sequence.
\end{lemma}
\begin{proof}
It is enough to show that $\go^1_c(\ol{P},V)_{\ol{\r}_V}/dC^\8_c(\ol{P},V)_{\ol{\r}_V}$ is sequentially complete. To this end we show that $\ps \co \go^1_c(\ol{P},V)_{\ol{\r}_V} \ra (\go_c^1(\ol{P},V)/ (dC^\8_c(\ol{P},V)))_\fix, ~ \o \ms [\o]$ is surjective and $\ker(\ps) = dC^\8_c(\ol{P},V)_{\ol{\r}_V}$ where $(\go_c^1(\ol{P},V)/ (dC^\8_c(\ol{P},V)))_\fix$ stands for the fixed points of the natural action of $\ol{H}$ on $\go_c^1(\ol{P},V)/ (dC^\8_c(\ol{P},V))$. The inclusion $dC^\8_c(\ol{P},V)_{\ol{\r}_V} \subs \ker(\ps)$ is clear. Let $n:=\#\ol{H}$, $[\o] \in (\go_c^1(\ol{P},V)/ (dC^\8_c(\ol{P},V)))_\fix$ and $[\o]=[df]$ for $f\in C^\8_c(\ol{P},V)$. Then $[\o] = [d \fr 1n \sum_{g\in \ol{H}} g.f]$. In order to show that $\ps$ is surjective let $\o \in \go^1_c(\ol{P},V)$ such that $[\o]$ is $\r_V$-invariant. Then $[\o] = [\fr 1n \sum_{g\in \ol{H}} g.\o]$ and $\fr 1n \sum_{g\in \ol{H}} g.\o \in \go^1_c(\ol{P},V)_{\ol{\r}_V}$. Hence 
\begin{align*}
\go^1_c(\ol{P},V)_{\ol{\r}_V}/(dC^\8_c(\ol{P},V)_{\ol{\r}_V}) \ra (\go_c^1(\ol{P},V)/ (dC^\8_c(\ol{P},V)))_\fix, ~ [\o] \ms [\o]
\end{align*}
is a continuous vector space isomorphism. It is also an isomorphism of topological vector spaces, because $(\go_c^1(\ol{P},V)/ (dC^\8_c(\ol{P},V)))_\fix \ra \go^1_c(\ol{P},V)_{\ol{\r}_V},~ [\o] \ms [\fr 1n \sum_{g \in \ol{H}} g.\o]$ is a continuous right-inverse. 
Hence $\go^1_c(\ol{P},V)_{\ol{\r}}/dC^\8_c(\ol{P},V)_{\ol{\r}}$ is sequentially complete if and only if 
\begin{align*}
(\go_c^1(\ol{P},V)/ (dC^\8_c(\ol{P},V)))_\fix 
\end{align*}
is sequentially complete.  
Because $\ol{q}$ is proper $\ol{P}$ is $\sigma$-compact.  For the same reason like in \cite[Chapter 2, p. 385]{Neeb:2009} we can assume $\ol{P}$ to be connected. 
Now Lemma \ref{GruFolgenvoll} shows that $\go_c^1(\ol{P},V)/ (dC^\8_c(\ol{P},V))$ is sequentially complete. Hence the assertion follows from
\begin{align*}
&(\go_c^1(\ol{P},V)/ (dC^\8_c(\ol{P},V)))_\fix = \bigcap_{g\in \ol{H}} \set{\o \in \go_c^1(\ol{P},V)/ (dC^\8_c(\ol{P},V)): \ol{\r}_V(g)\ci \ol{R}_g^\ast \o =\o}\\
=&\bigcap_{g\in \ol{H}} (\ol{\r}_V(g)\ci \ol{R}_g^\ast - \id)\inv\set{0}.
\end{align*}
\end{proof}

%---------Alt-Beginn--------

%---------Alt-Ende--------
\end{document}